\documentclass[12pt, a4paper]{amsart}

\usepackage{graphicx}

\usepackage{amscd,amsmath,amssymb,amsthm,amsfonts}
\usepackage[alphabetic]{amsrefs}
\usepackage{mathrsfs}
\usepackage[shortlabels]{enumitem}
\usepackage[all]{xy}

\usepackage{xcolor}
\usepackage[hypertexnames=true, citecolor = green, colorlinks = false, linkcolor = red, linktoc = none, pdffitwindow = false, urlbordercolor = white]{hyperref}%

\def\Aut{\operatorname{Aut}}
\def\Out{\operatorname{Out}}
\def\End{\operatorname{End}}

\def\Ad{\operatorname{Ad}}

\def\max{\operatorname{max}}

\def\N{\mathbb{N}}

\newcommand{\IC}[0]{\mathbb{C}}

 \newcommand{\IN}[0]{\mathbb{N}}
 
 \newcommand{\IR}[0]{\mathbb{R}}
 \newcommand{\IT}[0]{\mathbb{T}}

 \newcommand{\IZ}[0]{\mathbb{Z}}

\newcommand{\CA}[0]{\mathcal{A}} \newcommand{\CB}[0]{\mathcal{B}}
 \renewcommand{\CD}[0]{\mathcal{D}}
 \newcommand{\CF}[0]{\mathcal{F}}
 \newcommand{\CH}[0]{\mathcal{H}}
 
\newcommand{\CK}[0]{\mathcal{K}} 
 
\newcommand{\CO}[0]{\mathcal{O}} \newcommand{\CP}[0]{\mathcal{P}}
\newcommand{\CQ}[0]{\mathcal{Q}} 
\newcommand{\CS}[0]{\mathcal{S}} 
\newcommand{\CU}[0]{\mathcal{U}} 
\newcommand{\CW}[0]{\mathcal{W}}

\newtheorem{thm}{Theorem}[section]
\newtheorem{corollary}[thm]{Corollary}
\newtheorem{lemma}[thm]{Lemma}

\newtheorem{proposition}[thm]{Proposition}

\theoremstyle{definition}

\theoremstyle{remark}
\newtheorem{remark}[thm]{Remark}

\newtheorem{example}[thm]{Example}

\numberwithin{equation}{section}

\begin{document}

\title[Normalizers and permutative endomorphisms of $\CQ_2$]{Normalizers and permutative endomorphisms\\ of the $2$-adic ring $C^*$-algebra}

\author[V.~Aiello]{Valeriano Aiello}
\address{
Section de Math\'ematiques
Universit\'e de Gen\`eve
2-4 rue du Li\`evre, Case Postale 64, 1211 Gen\`eve 4, Suisse
}
\email{valerianoaiello@gmail.com}

\author[R.~Conti]{Roberto Conti}
\address{Dipartimento di Scienze di Base e Applicate per l'Ingegneria \\ Sapienza Universit\`{a} di Roma \\ Via A. Scarpa 16, 00161 Roma, Italy}
\email{roberto.conti@sbai.uniroma1.it}

\author[S.~Rossi]{Stefano Rossi}
\address{Dipartimento di Matematica \\ Universit\`{a} di Roma Tor Vergata \\ Via della Ricerca Scientifica, 1, 00133 Roma \\ Italy}
\email{rossis@mat.uniroma2.it}

\thanks{Valeriano Aiello acknowledges   support of the Swiss National Science Foundation. Roberto Conti acknowledges partial support by Sapienza Universit\`a di Roma.
Stefano Rossi is supported by European Research Council Advanced Grant 669240 QUEST}

\begin{abstract}
A complete description is provided for the unitary normalizer of the diagonal Cartan subalgebra $\CD_2$ in the $2$-adic
ring $C^*$-algebra $\CQ_2$, which generalizes and unifies analogous results for Cuntz and Bunce-Deddens algebras. 
Furthermore, the inclusion $\CO_2\subset\CQ_2$ is  proved not to be regular. Finally, countably many novel permutative endomorphisms
of $\CQ_2$ are exhibited with prescribed images of the generator $U$.

\end{abstract}

\date{\today}
\maketitle
\tableofcontents

\section{Introduction}

The study of automorphisms (endomorphisms) of $C^*$-algebras does not seem to have received
as much attention as its classical counterpart. Arguably, the groups (semigroups) made up of
automorphisms (endomorphisms) of a non-commutative $C^*$-algebra  are seldom regarded as inviting objects to deal with  
in that they are not only  difficult
to describe in concrete terms but they also lack many of those properties a group is generally supposed to possess.
For instance, these groups are hardly ever locally compact, apart from those coming from finite-dimensional
$C^*$-algebras. A few exceptions, however, do exist. A case in point is given by the Cuntz algebras $\CO_n$: their
endomorphisms and automorphisms have in fact been studied rather intensively despite the difficulties alluded to above, perhaps because of their  interplay with algebraic quantum field theory, whose superselection
structure can be phrased in terms of suitable equivalence classes of endomorphisms. 
Indeed, as late as over forty years after their introduction in \cite{Cun}, the Cuntz algebras  still attract much attention.
Unlike many other  $C^*$-algebras, this is particularly true of their endomorphisms and automorphisms  \cite{CuntzAut, CRS, CoSz11, CHS-crelle, CHS}. 
Another reason is they  
display a remarkably rich variety of phenomena  
which range from  the study of general structure properties of $C^*$-algebras to
dynamical systems, actions of (possibly quantum) groups and  subfactors. 
Moreover,  these endomomorphisms quite often lead to non-trivial computations of important invariants, such as the Jones index or Voiculescu's topological entropy \cite{Izumi, Longo, Jones, CP, Choda, Skalski}.   
Nevertheless, far less is known about the general structure of other $C^*$-algebras which might happen to be 
somewhat related to the Cuntz algebra,
and their endomorphism semigroup.
For instance,  not too long ago Cuntz and others introduced a vast class of $C^*$-algebras naturally associated with algebraic structures of various kind, which seem to indicate that intriguing 
connections are very likely to be found between operator algebra theory on the one hand and other seemingly far different areas, most notably number theory, on the other (see e.g. \cite{CELY17} for a wide overview).
Having that in mind, in our recent works \cite{ACR, ACR2, ACR3} we initiated a painstaking analysis  of the so-called dyadic $C^*$-algebra $\CQ_2$, which was first studied systematically by
Larsen and Li in \cite{LarsenLi}.
This $C^*$-algebra contains in a  canonical fashion both the Cuntz algebra $\CO_2$ and the Bunce-Deddens algebra of type $2^\infty$ (the latter as the fixed-point subalgebra  $\CQ_2^\IT$ of the gauge action).
 In particular, much of our attention so far has been lavished on its endomorphisms and automorphisms, of which very little was known before.  
To frame the scope and the reach of our analysis, however, it might be worth  stressing that $\CQ_2$ is not at all an isolated case. On the contrary,
it is perhaps best presented as a noticeable example of a broad class of $C^*$-algebras arising from algebraic dynamical systems, including all $\CQ_n$ with $n \geq 2$, which have been addressed in \cite{ACRS} in much greater generality than was initially done in \cite{ACR}.

\medskip
 
Our final goal was at that time and still is to arrive at a thorough description of the group $\Out(\CQ_2)$ of the outer automorphisms
of $\CQ_2$ not least because only rarely has such an ambitious task been accomplished.  
Even so,  the undertaking is not necessarily bound to fail. Indeed, unlike the Cuntz algebra, the $2$-adic ring $C^*$-algebra features a decidedly more rigid structure 
in that the Cuntz isometries are now intertwined, which in fact seems to prevent many cases from occurring. 
Many obstacles are easily found on the way, though, and this might depend on the various facets of $\Aut(\CQ_2)$ and ${\rm End}(\CQ_2)$ entailed
by the intricacies of the ladder of inclusions
$$
\begin{array}{ccccc}
\CQ_2 & \supset & \CQ_2^{\mathbb T}=C^*(\CD_2,U) & \supset & C^*(U) \\
\cup && \cup && \\
\CO_2 & \supset & \CF_2 & \supset & \CD_2 \\
\cup && && \\
C^*(S_2) &&&&
\end{array}
$$
which account for the deep interplay between $\CQ_2$ and the Cuntz algebra $\CO_2$.
Still far from a complete answer, we have nonetheless obtained partial yet motivating results. 
Among them, and without any pretense of exhaustiveness, we  showed that  $\End_{C^*(U)}(\CQ_2) = \Aut_{C^*(U)}(\CQ_2)\simeq C({\mathbb T},{\mathbb T})$ and $\Aut_{\CD_2}(\CQ_2)$ are both maximally Abelian in $\Aut(\CQ_2)$. As far as $\Out(\CQ_2)$ is concerned, at present all we know is it is uncountable and not Abelian.
Moreover, any extendible localized diagonal automorphism of $\CO_2$ is the product of a gauge and a localized diagonal inner automorphism.
Finally, we also spotted an interesting rigidity phenomenon relative to the inclusion $\CO_2\subset\CQ_2$, which forbids any non-trivial endomorphism of
$\CQ_2$ to restrict to $\CO_2$ trivially. Nevertheless, the inclusion, albeit given quite explicitly, is not easily handled with standard techniques, and the principal reason is there is no way to
see $\CO_2$ as a fixed-point subalgebra of $\CQ_2$, for no conditional expectation exists from the larger onto the smaller algebra.

\medskip
As for the present work, we start by going back to   the analysis of the automorphisms of $\CQ_2$  mapping $C^*(U)$ onto itself, for a number of problems had been left open in
\cite{ACR}. More precisely, our attention is here turned to those automorphisms which at the level of $C^*(U)$ simply act as a rotation of the generator. 
Remarkably, the only allowed values of the angle turn out to be all roots of order any power of $2$. 
We then move on to provide a complete description of those inner automorphisms of $\CQ_2$ leaving the diagonal $\CD_2$ globally invariant. This should be regarded as the main result of the present work insofar as it not only fully settles the problem but also establishes an elegant synthesis of the corresponding results for  the Bunce-Deddens and Cuntz algebras \cite{PUT,Pow}. 
Interestingly enough, the associated non-trivial part $\CW$ of the normalizing group  is a specific extension of the well-known Thompson group $V$, which have surfaced before in the work of Nekrashevych  \cite{NEK}.
Moreover, we draw a number of consequences of these results on the structure of other normalizers. In particular, the inclusion $\CO_2 \subset \CQ_2$ is proved not to be
regular, even though  to date we  do not know whether the unitary normalizer of $\CO_2$ in $\CQ_2$ actually reduces to $\CU(\CO_2)$, as we would be inclined to believe.
At any rate, we do show that the only unitaries in the Bunce-Deddens algebra normalizing $\CO_2$ are those in the canonical UHF subalgebra $\CF_2$.
 Finally, in the last part, which is more combinatorial in character,  we  discuss permutative endomorphisms of $\CO_2$, namely those obtained by extending the permutative endomorphisms of $\CO_2$. Extendability is no trivial matter here. Roughly speaking, if one is given an endomorphism of the Cuntz algebra, the odds are
it will fail to extend to an endomorphism of the whole $\CQ_2$. To take but two significant examples of the hurdles one might encounter in trying to
exhibit an extension  of any such endomorphism, it is worth recalling that less than half of the permutative endomorphisms of $\CO_2$ at level two actually extend \cite{ACR3},
and among the so-called Bogolubov automorphisms of $\CO_2$ only the gauge automorphisms, the flip-flop and their compositions extend \cite{ACR}.
That being the case, one might be led to expect extendible endomorphisms to be increasingly sparse as the level is raised. 
Quite the opposite, we show that the number of permutative endomorphisms of $\CQ_2$ does grow extremely quickly with the level, which
came as good news to us. Moreover, we are now in a position to  enrich the list of the endomorphisms
of $\CQ_2$, which admittedly had  remained rather limited since we started working on the problem in \cite{ACR}.

\section{Preliminaries and notation}

This rather quick section provides the reader with the basic notation and definitions needed to make the paper as self-consistent and readable as possible.
The main object of the present study,  the $2$-adic ring $C^*$-algebra $\CQ_2$ is by definition  the universal $C^*$-algebra generated by a unitary $U$
and a (proper) isometry $S_2$ such that $S_2U=U^2S_2$ and $S_2S_2^*+US_2S_2^*U^*=1$.
Several characterizations of this $C^*$-algebra are actually known, see \cite{LarsenLi} for more detail.
Among the many interesting properties enjoyed by $\CQ_2$, it is worth recalling it is a simple and purely infinite $C^*$-algebra.
As is known, the Cuntz algebra $\CO_2$ is the universal $C^*$-algebra generated by two isometries $X_1, X_2$ such that
$X_1X_1^*+X_2X_2^*=1$.  
It is clear that $\CO_2$ embeds into $\CQ_2$ through the injective $^*$-homomorphism that sends $X_1$ to $US_2$ and $X_2$ to $S_2$.
A distinguished representation of the $2$-adic ring $C^*$-algebra, which will actually play a major role in this work, is the so-called canonical representation $\rho_c$ 
of $\CQ_2$, which 
is a (faithful) irreducible representation acting on the Hilbert space
$\ell_2(\IZ)$, with canonical orthonormal basis $\{e_k:k\in\IZ\}$,
 by $\rho_c(U)e_k\doteq e_{k+1}$   and $\rho_c(S_2)e_k\doteq e_{2k}$, $k\in\IZ$. 
In order to ease the notation, we will often drop the symbol $\rho_c$ and identify $\CQ_2$ with its image. 
The Cuntz algebra $\CO_2$ is acted upon by the one-dimensional torus $\IT$ through the well-known gauge automorphisms $\alpha_\theta$, with $\theta\in\IR$.
These are given by $\alpha_\theta (S_i)=e^{i\theta}S_i$  for $i=1,2$.    The corresponding invariant subalgebra is denoted by $\CF_2\subset \CO_2$, which will often be referred to as the gauge invariant subalgebra of 
$\CO_2$. It is worth mentioning that $\CF_2$ is isomorphic with the unique UHF algebra of type $2^{\infty}$.
 Now the gauge automorphisms extend to automorphisms $\widetilde{\alpha_\theta}$ of the whole $\CQ_2$ by setting $\widetilde{\alpha_\theta}(U)=U$, which
allows us to consider the gauge invariant subalgebra $\CQ_2^\mathbb{T}$ of $\CQ_2$ as well. Among other things,  $\CQ_2^\IT$ is known to be a Bunce-Deddens algebra. 
It is not difficult to see that $\CQ_2^{\mathbb{T}}$ can also be described 
as the $C^*$-subalgebra of $\CQ_2$ generated from $\CD_2$ and $U$, where $\CD_2\subset\CF_2$ is the diagonal subalgebra, namely
the subalgebra generated by the diagonal projections $P_\alpha\doteq S_\alpha S_\alpha^*$, where for any multi-index $\alpha=(\alpha_1,\alpha_2,\ldots,\alpha_k)\in W_2\doteq \bigcup_{n\geq 0} \{1,2\}^n$ the isometry $S_\alpha$ is the product $S_{\alpha_1}S_{\alpha_2}\ldots S_{\alpha_k}$. 
Sometimes it will be more convenient to identify $\CD_2$ with the $C^*$-algebra of continuous functions on its Gelfand spectrum, which is known to be the Cantor set
$K=\{1,2\}^\IN$. As a matter of fact, $\CD_2$ is a Cartan subalgebra both of $\CO_2$ and $\CQ_2$.
Since the multi-index notation  will be adopted extensively throughout the paper, we take this opportunity to  recall that $|\alpha|$ denotes the length of the multi-index $\alpha$.

The canonical endomorphism of $\CO_2$  
is defined on each element $x\in\CO_2$ as $\varphi(x)=S_1xS_1^*+S_2xS_2^*$. 
It is rather obvious that it extends to $\CQ_2$. 
We also point out the intertwining rules $S_i x = \varphi(x)S_i$ for every  $x\in {\mathcal Q}_2$, with $i=1,2$, which will come in useful in the sequel.  
Lastly, we recall that thanks to the Cuntz-Takesaki correspondence, every endomorphism of $\CO_2$ is uniquely determined by a unitary  in $\CO_2$. To be precise, given $u\in\CU(\CO_2)$ there exists an
endomorphism $\lambda_u$ defined as $\lambda_u(S_i)\doteq uS_i$ for $i=1, 2$ and conversely every endomorphism has this form.
Aware that this overview can by no means be regarded as a comprehensive introduction, we refer the interested reader to \cite{ACR, LarsenLi}, and the references therein, for a fuller coverage of the material instead.

\section{Automorphisms preserving $C^*(U)$}

The present brief section aims to refine some results concerning the $C^*$-subalgebra of $\CQ_2$ generated by $U$.
To begin with, in \cite{ACR} the commutative subalgebra $C^*(U)$ was proved to be maximal Abelian in $\CQ_2$, and it was also seen to be the image of a unique conditional
expectation from $\CQ_2$. However, our subalgebra fails to be a Cartan subalgebra, as shown by the following result.

\begin{proposition}
The normalizer  
$$
\{	v\in\CQ_2 \; | \; v^*v, vv^*\in\CP(\CQ_2),\; vC^*(U)v^*\subset C^*(U), \; v^*C^*(U)v\subset C^*(U)		\}
$$
coincides with the unitary normalizer $N_{C^*(U)}(\CQ_2)$ and  
 sits in the Bunce-Deddens subalgebra $\CQ_2^\IT$.
In particular, the subalgebra $C^*(U)$ is not  Cartan in $\CQ_2$. 
\end{proposition}
\begin{proof}

Let $v$ be a partial isometry in the normalizer. The projections $v^*v$ and $vv^*$ are both in $C^*(U)$.
Since $\IT$ is connected, they can only be both equal to $0$ or $1$. Therefore, without loss of generality we may suppose that 
$v$  is a unitary in $N_{C^*(U)}(\CQ_2)$. By definition we have $vUv^*=g(U)$ for some $g\in C(\IT)$. If we apply the gauge automorphism $\alpha_z$ to the previous equality we get $\alpha_z(v)U\alpha_z(v)^*=g(U)$, which leads to $v^*\alpha_z(v)U\alpha_z(v)^*v=U$. The maximality of $C^*(U)$ in $\CQ_2$ implies that $\alpha_z(v)=vh_z(U)$ for some function $h_z\in C(\IT)$. We observe that  $\alpha_{z_1z_2}(v)=vh_{z_1z_2}(U)=\alpha_{z_1}(\alpha_{z_2}(v))=\alpha_{z_1}(vh_{z_2}(U))=\alpha_{z_1}(v)\alpha_{z_1}(h_{z_2}(U))=vh_{z_1}(U)h_{z_2}(U)$ and thus $h_{z_1z_2}(U)=h_{z_1}(U)h_{z_2}(U)$. Set $f_w(z):=h_z(w)$ where $w\in\IT$. The function $f_w(\cdot)$ is a continuous function on $\IT$ (as a function in the variable $z$)  which is also a character. Indeed, it holds
$f_w(z_1z_2)=h_{z_1z_2}(w)=h_{z_1}(w)h_{z_2}(w)=f_w(z_1)f_w(z_2)$. It follows that $f_w(z)=z^{k(w)}$ for some $k(w)\in\IZ$. 
Now $k(w)$ is the winding number of the curve $f_w(z)$ (where $w$ is fixed). All the curves   $\{f_w(\cdot)\}_{w\in\IT}$ are homotopic. Since the winding number is homotopy invariant, 
we see that $k(w)$ 
has to be constant, say $k$.
Now,   a straightforward argument 
   shows that $k$ has to be equal to $0$, cf. Proposition \ref{normalizer}.
The claim about the Cartan subalgebra is obvious.
\end{proof}

In \cite{ACR} it was shown that corresponding to any root of unity $z$ of order a power of $2$ there existed an inner automorphism $\Ad(U_z)$, implemented
by a unitary $U_z\in\CD_2$, such that
$\Ad(U_z)(U)=zU$. Our guess was that automorphisms of this sort should cease to exist as soon as $z$ was no longer such a root. 
What the next result does is bridge this gap and show that no automorphism of $\CQ_2$ can send $U$ to $zU$ unless $z$ is a root of unity of order a power of two.

\begin{proposition}
Let $\alpha$ be an automorphisms of $\CQ_2$ such that $\alpha(U)=zU$ for some $z\in\IT$, then $z^{2^n}=1$ for some $n$.
\end{proposition}
\begin{proof}
We already know that if $z^{2^n}=1$ for some $n$, then there exists an automorphism mapping $U$ to $zU$. Indeed, one can consider   $\Ad(U_z)$, where $U_z$ is the unitary in the diagonal subalgebra $\CD_2$ defined in \cite[Section 6.3]{ACR} by the formula $U_z e_k\doteq z^k e_k$ for all $k\in \IZ$. Suppose that $z$ has order different from $2^n$ for all $n$. 
From now on $\CQ_2$ will be understood in the interval picture, we refer to \cite{ACRS} for the definition of such representation. 
Denote by $1\in L^2([0,1])$ the unit constant function. 
It is easy to see that $1$ is an eigenvector for $U$.
Now for any 
 multi-index $\alpha$, the vector $v_\alpha := P_\alpha 1$ is an eigenvector for $U^{2^{|\alpha |}}$. Indeed, we have $U^{2^{|\alpha |}}(P_\alpha 1)=P_\alpha U^{2^{|\alpha |}}1=P_\alpha 1$. The family of vectors $\{v_\alpha\}_{\alpha\in W_2}$ is 
 a complete system
  for $L^2([0,1])$. 
If there existed an automorphism mapping $U$ to $zU$, then $U^2$ would have $\bar z$ as an eigenvalue. Indeed, let $\alpha$ be such an automorphism. Then
$\alpha(U^2)\alpha(S_2)1=z^2U^2\alpha(S_2)1=\alpha(S_2)\alpha(U)1=z\alpha(S_2)1$ which shows that $\alpha(S_2)1$ is an eigenvector for $U^2$ with eigenvalue $\bar z$. As $z$ has order different from $2^n$ for all $n\in \IN$, the vector $\alpha(S_2)1$ would be orthogonal to all the $v_\alpha$, which is absurd. 
\end{proof}

\section{The normalizer of $\CD_2$ in $\CQ_2$}

The normalizer of the diagonal $\CD_2$ has been completely described in both the Cuntz algebra $\CO_2$ and the Bunce-Deddens algebra $\CQ_2^{\IT}$.
More precisely, $N_{\CD_2}(\CQ_2^\IT) = \CU(\CD_2) \cdot
\big\{ u\in \CQ_2^\IT \; | \; u=\sum_{i\in F} p_i U^i, \; F \subset \IZ, \, |F|<\infty , \ p_i^2=p_i^*=p_i \in\CD_2, \forall i \in F , \sum_{i \in F} p_i = 1 =  \sum_{i\in F} \Ad(U^{-i})(p_i) \big\}$, 
see \cite[Lemma 5.1]{PUT}, while $N_{\CD_2}(\CO_2) = \CU(\CD_2) \cdot \CS_2$ \cite{Pow},  where $\CS_2$ is the group of unitaries in $\CO_2$ that can be written as a finite sum of words in the generators $S_i$, $i=1,2$ and their adjoints, cf. \cite{CoSz11}.
In this section we generalize both these results.

The next proposition provides a good many examples of unitary in $\CQ_2$ normalizing the diagonal subalgebra $\CD_2$.
Before stating it, a couple of points are needed. First,  any monomial $S_\alpha S_\beta^*U^k$, $\alpha,\beta\in W_2$ and $k\in\mathbb{Z}$, can be rewritten
as $S_\alpha U^l S_\gamma^*$, for suitable $l\in\mathbb{Z}$ and $\gamma\in W_2$ depending on $k$ and $\beta$.  Moreover, the latter representation is more convenient  not only
because it is symmetric under taking the adjoint but because it is also canonical insofar as it is unique.  In other terms, the equality $S_\alpha U^k S_\beta^*=S_{\alpha'}U^{k'}S_{\beta'}^*$ is possible only if $\alpha=\alpha'$, $\beta=\beta'$ and $k=k'$.

\begin{proposition}\label{unitaryproj}
Given a finite family of triples $(\alpha_i,\beta_i, k_i)\in W_2\times W_2\times \mathbb{Z}$, with $i=1,2,\ldots, N$, define
$u\doteq \sum_{i=1}^N S_{\alpha_i}U^{k_i}S_{\beta_i}^*\in\CQ_2$.
The element $u$ is unitary if and only if $\sum_{i=1}^N S_{\alpha_i}S_{\alpha_i}^*=\sum_{i =1}^N S_{\beta_i}S_{\beta_i}^*=1$. In that case, the unitary $u$ also
belongs to $N_{\CD_2}(\CQ_2)$.
\end{proposition}

\begin{proof}
We begin by observing that each summand appearing in the sum that defines $u$ is a partial isometry. More precisely, any element
of the form $S_\alpha U^k S_\beta^*$ is a partial isometry whose initial and final projections are $S_\beta S_\beta^*$ and $S_\alpha S_\alpha^*$
respectively. Therefore, the condition  $\sum_{i=1}^N S_{\beta_i}S_{\beta_i}^*=1$ guarantees that $u$ is a full isometry, while the condition 
 $\sum_{i=1}^N S_{\alpha_i}S_{\alpha_i}^*=1$ says that $u$ is in addition a surjective isometry. This obviously proves the if part.
 Conversely, let us assume that $u$ is unitary. The first thing we need to prove is that the projections $S_{\beta_i} S_{\beta_i}^*$ are pairwise orthogonal.
 If we work in the canonical representation of $\CQ_2$, this amounts to showing that given $n\in\mathbb{Z}$ such that $S_{\beta_l}S_{\beta_l}^*e_n\neq 0$ and 
 $S_{\beta_m}S_{\beta_m}^*e_n\neq 0$ then $l=m$. But if this were not the case, we would find the absurd inequality
\begin{align*}
 1&=\|ue_n\|= \big\|\sum_{i=1}^N S_{\alpha_i}U^{k_i}S_{\beta_i}^* e_n\big\| \\
&=\big\| S_{\alpha_l}U^{k_l}S_{\beta_l}^*e_n+ S_{\alpha_m}U^{k_m}S_{\beta_m}^*e_n+ \sum_{i\neq l,m} S_{\alpha_i}U^{k_i}S_{\beta_i}^*e_n\big\|\geq \sqrt{2} 
 \end{align*}
 where the last inequality is due to the fact that every non-zero term is a basis vector.
 This clearly shows that $\sum_{i=1}^N S_{\beta_i}S_{\beta_i}^*\leq 1$. By applying the same argument to $u^*$ we see that the inequality 
 $\sum_{i=1}^N S_{\alpha_i}S_{\alpha_i}^*\leq 1$ holds as well. At this point, it is clear that both inequalities must actually be equalities, for otherwise
 $u$ could not be unitary.\\
 To conclude, we have to show that $u$ lies in the normalizer of $\CD_2$ in $\CQ_2$. To this aim, it is enough to 
 verify that both $uS_\gamma S_\gamma^*u^*$ and $u^*S_\gamma S_\gamma^*u$ are still in $\CD_2$ for every multi-index $\gamma$. We only deal
 with the first term, for the second is handled in the very same fashion. Now
 \begin{align*}
 uS_\gamma S_\gamma^*u^*&= \Big(\sum_{i=1}^N S_{\alpha_i}U^{k_i}S_{\beta_i}^*\Big)S_\gamma S_\gamma^*\Big(\sum_{j=1}^N S_{\alpha_j}U^{k_j}S_{\beta_j}^*\Big)^*\\
 &=\sum_{i,j=1}^N S_{\alpha_i} U^{k_i} \big(S_{\beta_i}^*S_\gamma S_\gamma^* S_{\beta_j}\big)U^{-k_j} S_{\alpha_j}^*
 \end{align*}
 
Now there is no lack of generality if we further assume that the length of $\gamma$ is greater of ${\rm max}\{|\alpha_i|, |\beta_i|:i=1,2,\ldots , N\}$.
In this case, the only way for a term of the form $S_{\beta_i}^*S_\gamma S_\gamma^* S_{\beta_j}$ not to be zero is $i=j$.  
In particular, in the above sum only one term survives, which means $uS_\gamma S_\gamma^*u^*= 
S_{\alpha{i_\gamma}}U^{k_{i_\gamma}} S_{\beta_{i_\gamma}}^*S_\gamma S_\gamma^*S_{\beta_{i_\gamma}}U^{-k_{i_\gamma}}S_{\alpha_{i_\gamma}}^*\in\CD_2$.
\end{proof}
It is worth stressing that the powers of $U$ occurring in the sums above can be chosen arbitrarily. 

\begin{remark}
All powers of $U$ can be recovered as particular instances of the above unitaries. More precisely, they correspond to the case $N=1$. 
\end{remark}

The unitaries yielded by the above proposition clearly form a group $\mathcal{W}$ that contains the Thompson group $\mathcal{S}_2\cong V$. 
It is worthwhile to observe that  a natural class of irreducible unitary representations of $\CW$ can be obtained
by restricting irreducible representations of $\CQ_2$ to it.  
cf. \cite{Haag}. Among these, permutative irreducible representations of $\CQ_2$ represent quite an interesting class of examples inasmuch as 
they have been thoroughly classified in \cite{ACR3}.
For the sake of completeness, a diagrammatic description of $\mathcal{W}$ will be outlined in the next section.\\    

There follows a series of technical results necessary to reach the main theorem of this section.

\begin{lemma}\label{orthonormal}
Given a multi-index $\beta\in W_2$, let $A_\beta\subset\mathbb{Z}$ be the set $\{k\in\mathbb{Z}: e_k=P_\beta e_k\}$.
For any finite set of distinct monomials of the form $S_{\alpha_i}U^iS_\beta^*\in \CB(\ell_2(\mathbb{Z}))$, with $i=1,2,\ldots , l$, there exists at least an 
$m\in A_\beta$ such that $\{S_{\alpha_i}U^iS_\beta^* e_m: i=1,2,\ldots , l\}$ is an orthonormal system. 
\end{lemma}

\begin{proof}
Given any $i\in\{1,2,\ldots , l\}$, let $f_i: A_\beta\rightarrow \mathbb{Z}$ be the function such that 
$S_{\alpha_i}U^i S_\beta^* e_k = e_{f_i(k)}$ for every $k\in A_\beta$.
Each of this function is of the form $f_i(k)=a_ik+b_i$, $k\in A_\beta$, where $a_i, b_i$ are suitable integers.
Because the monomials $S_{\alpha_i}U^iS_\beta^*$ are distinct, the functions $f_i$ are all distinct as well, which means
the equation $f_i(k)=f_j(k)$, $k\in A_\beta$, can only have at most one solution for any pair $(i,j)$ with
$i\neq j$.
Let $C_{i,j}\subset\mathbb{Z}$ be the set defined as $\{k\in A_\beta: f_i(k)=f_j(k)\}$
The set $C\doteq\cup_{i\neq j} C_{i,j}$ is finite and its cardinality is clearly not greater than
$\frac{l(l-1)}{2}$. In particular, its complement $D$ in $A_\beta$ is not empty (actually it is infinite). The conclusion now follows by
noting that $D$ is nothing but the set $\{k\in A_\beta: f_1(k)\neq f_2(k)\neq\ldots\neq f_l(k)\}$.
\end{proof}

By a permutative unitary we mean any unitary operator acting on $\ell_2(\IZ)$ permuting the elements of the canonical basis.
\begin{proposition}
Let $V$ be a permutative unitary in $\CQ_2$ and let $0<\varepsilon<1$  
and $\gamma_1, \gamma_2,\ldots , \gamma_N\in\mathbb{C}$.
If 
$$
\|V- \big(\gamma_1S_{\alpha_1}U^{k_1}S_\beta^*+\ldots+ \gamma_l S_{\alpha_l}U^{k_l}S_\beta^* +\sum_{i=l+1}^N \gamma_i S_{\alpha_i}U^{k_i}S_{\beta_i}^* \big)\|<\varepsilon
$$
and $P_\beta\perp P_{\beta_i}$ for every $i=l+1,\ldots , N$, then there exists an $i_0\in\{1,2,\ldots , l\}$ such that
$|1-\gamma_{i_0}|^2+\sum_{i\in\{1,2,\ldots , l\}\setminus\{i_0\}} |\gamma_i|^2<\varepsilon^2$.

\end{proposition}
\begin{proof}
Pick a $k_0$ in $D\subset A_\beta$ where $\{S_{\alpha_i}U^{k_i}S_\beta^*e_{k_0}: I=1,2,\ldots , l\}$ is an orthonormal system.

\begin{align*}
&\|Ve_{k_0}- \big(\gamma_1S_{\alpha_1}U^{k_1}S_\beta^*+\ldots+ \gamma_l S_{\alpha_l}U^{k_l}S_\beta^* +\sum_{i=l+1}^N \gamma_i S_{\alpha_i}U^{k_i}S_{\beta_i}^* \big)e_{k_0}\|^2=\\ 
& \|Ve_{k_0}- \big(\gamma_1S_{\alpha_1}U^{k_1}S_\beta^*+\ldots+ \gamma_l S_{\alpha_l}U^{k_l}S_\beta^*)e_{k_0}\|^2<\varepsilon^2
\end{align*}
Since $\varepsilon<1$, the inequality is satisfied only if there exists (a unique) $i_0\in\{1,2,\ldots , l\}$ such that
$S_{\alpha_{i_0}}U^{k_{i_0}}S_\beta^*e_{k_0}=Ve_{k_0}$. But then 

$$
|1-\gamma_{i_0}|^2+\sum_{i\in\{1,2,\ldots  , l\}\setminus\{i_0\}} |\gamma_i|^2=\|Ve_{k_0}- \big(\gamma_1S_{\alpha_1}U^{k_1}S_\beta^*+\ldots+ \gamma_l S_{\alpha_l}U^{k_l}S_\beta^*)e_{k_0}\|^2<\varepsilon^2
$$
as maintained. 
\end{proof}

\begin{corollary}\label{straightmonomial}
With the same hypotheses as above, if $0<\varepsilon<\frac{1}{2}$ then the equality $S_{\alpha_{i_0}}U^{k_{i_0}}S_\beta^*e_k=Ve_k$ holds for every $k\in A_\beta$ apart from a finite set.
\end{corollary}

\begin{proof}
Under the condition on $\varepsilon$ the equality is clearly satisfied for every $k\in D$, which by definition is the set 
$\{k\in A_\beta: S_{\alpha_1}U^{k_1}S_\beta^*e_k\neq\ldots\neq S_{\alpha_l}U^{k_l}S_\beta^*e_k\}$, whose complement in $A_\beta$ was shown to be finite in the proof of Lemma \ref{orthonormal}. Indeed, suppose that $S_{\alpha_{i_0}}U^{k_{i_0}}S_\beta^*e_{k_0}=Ve_{k_0}$ and $S_{\alpha_{i_1}}U^{k_{i_1}}S_\beta^*e_{k_1}=Ve_{k_1}$ for some $i_0\neq i_1$ and $k_0\neq k_1$. This would imply that $|1- \gamma_{i_0}|<1$,
$|1- \gamma_{i_1}|<1$, $|\gamma_{i_0}|<1/2$, $|\gamma_{i_1}|<1/2$ which are clearly incompatible. 
\end{proof}

We are now in a position to prove a result that gives a simple description of all permutative unitaries of $\ell_2(\mathbb{Z})$ which are also elements of $\CQ_2$.

\begin{thm}\label{permutativeQ2}
A permutative unitary $V\in\mathcal{B}(\ell_2(\mathbb{Z}))$ belongs to $\CQ_2$ if and only if it is of the form
$\sum _{i\in F} S_{\alpha_i}U^{k_i}S_{\beta_i}^*$, where $F$ is a finite set over which the triples $(\alpha_i,\beta_i, k_i)\in W_2\times W_2\times \mathbb{Z}$ run, with
$\sum_{i\in F} S_{\alpha_i}S_{\alpha_i}^*=\sum_{i\in F} S_{\beta_i} S_{\beta_i}^*=1$.
\end{thm}

\begin{proof}
Clearly, we only need to worry about the ''only if'' part.
Let $V$ be a permutative unitary in $\CQ_2$ and let $\Psi$ be the bijection of $\mathbb{Z}$ implementing $V$, i.e.
$Ve_k=e_{\Psi(k)}$, $k\in\mathbb{Z}$.
By definition, $V$ is a norm limit of a sequence $\{T_n: n\in\mathbb{N}\}$ of operators of the form
$T_n=\sum_{(\alpha_i, \beta_i, k_i)\in F_n} \gamma_i S_{\alpha_i} U^k_i S_{\beta_i}^*$, where $\gamma_i$ are all scalar coefficients, and $F_n\subset W_2\times W_2\times \mathbb{Z}$ is a finite set. 
To begin with, we observe that the inequalities $\sum_{i\in F_n} S_{\alpha_i} S_{\alpha_i}^*\geq 1$ and $\sum_{i\in F_n} S_{\beta_i} S_{\beta_i}^*\geq 1$ hold eventually
as $T_n$ is eventually an invertible operator. 
The case in which $\sum_{i\in F_n} S_{\beta_i} S_{\beta_i}^*=1$, for some $n$, is immediately dealt with, for $V$ is simply given by $\sum_{i\in F_n} S_{\alpha_i}U^{k_i}S_{\beta_i}^*$.
Indeed, in this case for every $k$ there exists a unique $i_0= i_0(k)$ such that $S_{\beta_{i_0}}^*e_k$ is different from zero. This means only one term survives
in the sum $\sum_{i\in F_n}\gamma_i S_{\alpha_i}U^{k_i} S_{\beta_i}^*e_k$, namely that corresponding to $i=i_0(k)$. So
for the inequality $\|e_{\Psi(k)}-\sum_{i\in F_n}\gamma_i S_{\alpha_i}U^{k_i}S_{\beta_i}^*e_k\|<\varepsilon$ to hold for any $k$ is necessary
that $S_{\alpha_{i_0(k)}} U^{i_0(k)} (S_{\beta_{i_0(k)}})^* e_k= e_{\Psi(k)}$ for every $k\in\mathbb{Z}$ as long as $\varepsilon$ is chosen strictly less
than $1$. Therefore, for every $k\in\mathbb{Z}$ we have $Ve_k = e_{\psi(k)}=\sum_{i\in F_n} S_{\alpha_i} U^{k_i} S_{\beta_i}^*e_k$.
The conclusion is now got to, as the equality $\sum_{i} S_{\alpha_i} S_{\alpha_i}^*=1$ is automatically satisfied thanks to Proposition \ref{unitaryproj}. 
In order to deal with the case in which the sum $\sum_{i\in F_n} S_{\beta_i} S_{\beta_i}^*$ is greater than
$1$, 
it is convenient to assume that for any given $n$ the lengths  of the multi-indices $\beta_i$ are all the same as $i$ runs over $F_n$, say $|\beta_i|=k$ for every $i$ ($k$ will of course depend on $n$).
Fix an $n$ such that $T_n$ is invertible and $\|T_n -V\|<\frac{1}{2}$.
Let us simply denote $T_n$ by $T=\sum_{i\in F} \gamma_i S_{\alpha_i}U^{k_i}S_{\beta_i}^*$, with $|\beta_i|=k$ for every $i\in F$,  to ease the notation.
Now if $\sum_{i\in F} S_{\beta_i} S_{\beta_i}^*$ is greater than $1$, then the ranges of the projections $P_{\beta_i}= S_{\beta_i} S_{\beta_i}^*$ overlap.
But because the length of all monomials $S_{\beta_i}$ is the same, the ranges of $P_{\beta_i}$ and $P_{\beta_j}$
may overlap only if $\beta_i=\beta_j$. If, for any fixed $\beta\in W_2^k\doteq \{1,2\}^k$, we define $F_\beta\doteq \{i: \beta_i=\beta\}\subset F$, then $T$ may be more suitably rewritten as $T=\sum_{\beta\in W_2^k} \sum_{F_\beta} S_{\alpha_i}U^{k_i}S_\beta^* $
(it is understood that if $F_\beta$ is empty the corresponding term is zero).
 Corollary \ref{straightmonomial} now says that for every $\beta$ there exists a unique $i_0= i_0(\beta)\in F_\beta$ such that
$S_{\alpha_{i_0}}U^{k_{i_0}}S_\beta^* e_k=Ve_k$ for every $k\in A_\beta$ apart from a finite set of $C_\beta\subset A_\beta$.
Let us now set  $T_\beta\doteq S_{\alpha_{i_0(\beta)}}U^{k_{i_0(\beta)}}S_\beta^*$ and let $R_\beta$ be the finite-rank operator given by $R_\beta e_k=Ve_k-T_\beta e_k$ if $k\in C_\beta$ and $Re_k=0$ otherwise. Then we have proved the equality
$V=\sum_{\beta\in W_2^k} (T_\beta+R_\beta)=R+\sum_{\beta\in W_2^k} T_\beta$, where $R$ is the sum of all $T_\beta$'s.
Since both $V$ and $\sum_{\beta\in W_2^k}$ are in $\CQ_2$, the operator $R$ is in $\CQ_2$ as well. But because $\CQ_2$ is simple, the intersection $\CQ_2\cap \mathcal{K}(\mathcal{H})$ is trivial. Therefore, $T$ must be zero, that is $V=\sum_{\beta\in W_2^k} T_\beta$, which ends our proof. 
\end{proof}

Before stating our main result, we still need to prove a preliminary result which has an interest in its own although  it should       be a well-known fact. To the best of our knowledge, however, it is nowhere remarked explicitly, which is why we include a proof.

\begin{proposition}\label{normdiag}
If $V$ is a unitary on $\ell_2(\mathbb{Z})$ such that $\Ad(V)(\ell_\infty(\mathbb{Z}))=\ell_\infty(\mathbb{Z})$, then $V$
uniquely decomposes as $V=dP$, where $d$ is a diagonal unitary, i.e. $d\in\ell_\infty(\mathbb{Z})$, and $P$ is a permutative unitary, i.e. $Pe_k=e_{\Psi(k)}$, for every $k\in\mathbb{Z}$, for a suitable bijection $\Psi$ of $\mathbb{Z}$. 
\end{proposition}

\begin{proof}
We denote by $\delta_k\in\ell^\infty(\mathbb{Z})$ the orthogonal projection onto $\mathbb{C}e_k$. Since
$\Ad(V)$ restricts to  an automorphism of the von Neumann algebra $\ell_\infty(\mathbb{Z})$, $V\delta_kV^*$ must be a minimal projection of $\ell_\infty(\mathbb{Z})$. This means that $V\delta_kV^*=\delta_{\Psi(k)}$, for a suitable bijection
$\Psi$ of $\mathbb{Z}$ into itself, that is $V\delta_k=\delta_{\Psi(k)}V$. If we now evaluate the last equality on the vector $e_k$, we find $Ve_k=\delta_{\Psi(k)}Ve_k$. In other words, $Ve_k$ must be an eigenvector of $\delta_{\Psi(k)}$, and so
$Ve_k=\mu_ke_{\Psi(k)}$, where each $\mu_k$ is a complex number whose absolute value is one. If we define $d_k\doteq \mu_{\Psi^{-1}(k)}$, then $V$ can be rewritten as the product
$dP$, where $d\in\ell_\infty(\mathbb{Z})$ is the diagonal operator whose action 
on the basis vectors is given by $de_k\doteq d_ke_k $ and $P$ the permutative associated with $\Psi$.
Finally, the uniqueness of this decomposition is entirely obvious.
\end{proof}

\begin{thm}
Any unitary $v\in\CQ_2$ that normalizes the diagonal $\CD_2$ can be uniquely
written as $v=dP$, where $d$ is a unitary belonging to $\CD_2$ and $P\in\CQ_2$ a unitary of the form
$\sum_{i=1}^N S_{\alpha_i}U^{k_i}S_{\beta_i}^*$, with $\sum_{i=1}^N S_{\alpha_i}S_{\alpha_i}^*=\sum_{i=1}^N S_{\beta_i}S_{\beta_i}^*= 1$. 
\end{thm}

\begin{proof}
From now till the end of the proof we will be working in the canonical representation of $\CQ_2\subset\CB(\ell_2(\mathbb{Z}))$. 
The uniqueness of the decomposition is pretty obvious as the sole operator which is simultaneously diagonal and permutative (with respect to the canonical basis of $\ell_2(\mathbb{Z})$) is the identity $1$.
Let us now go to the existence of such a decomposition. We first note that if a unitary $V\in\CB(\ell_2(\mathbb{Z}))$ normalizes
$\CD_2$, then it also normalizes the von Neumann algebra generated by it, namely $\ell_\infty(\mathbb{Z})$. 
In light of Proposition \ref{normdiag}, our unitary
$V$ factors as a product $dP$, where $d$ is a diagonal unitary, that is $d\in\ell_\infty(\mathbb{Z})$ and $P$ is a permutative unitary, that is there exists a bijection $\Psi$ of $\mathbb{Z}$  such that 
$Pe_k=e_{\Psi(k)}$, for every $k\in\mathbb{Z}$. Now the proof boils down to showing that $d$ and $P$ actually sit in
$\CQ_2$ as a consequence of $V$ being a unitary of $\CQ_2$.\\
Because $dP$ lies in $\CQ_2$, for any $\varepsilon>0$ there exists an algebraic element $T_\varepsilon$ of the form
$\sum_{i=1}^{N_\varepsilon} \gamma_i S_{\alpha_i}U^{k_i}S_{\beta_i}^*$, where the $\gamma_i$'s are all complex coefficients, such that
$\|dP-T_\varepsilon\|<\varepsilon$. As soon as $\varepsilon$ is small enough, the operator $T_\varepsilon$ is invertible itself, which means the sums $\sum S_{\alpha_i}S_{\alpha_i}^*$ and $\sum S_{\beta_i}S_{\beta_i}^*$ are both greater than $1$.
We now want to rid ourselves of possible overlappings in much the same way as we did in the proof of Theorem \ref{permutativeQ2}. So suppose there is an $l$-tuple of overlapping terms. We do not harm the generality if we further suppose
these are just the first $l$ terms. In other terms, our $T_{\varepsilon}$ is of the form
$$\sum_{i=1}^l \gamma_iS_{\alpha_i}U^{k_i}S_\beta^*+\sum_{i=l+1}^{N_\varepsilon}  \gamma_iS_{\alpha_i}U^{k_i}S_{\beta_i}^* $$
with $P_\beta\perp P_{\beta_i}$ for every $i=l+1,\ldots, N_{\varepsilon}$. A very minor variation of the proof
of Corollary \ref{straightmonomial} tells us that there exists a subset $I_\beta\subset A_\beta\doteq \{k\in\mathbb{Z}: e_k=P_\beta e_k\}$ whose complement in $A_\beta$ is finite and for a unique $i_0\in\{1,2,\ldots, l\}$ one has
$S_{\alpha_{i_0}}U^{k_{i_0}}S_\beta^*e_k=e_{\Psi(k)}$ for every $k\in I_\beta$. Let us now define the finite-rank operator $R_\beta$ as $R_\beta e_k= dPe_k$ if $k\in A_\beta\setminus I_\beta$ and $R_\beta e_k=0$ otherwise.
The new operator $T_\varepsilon'\doteq \gamma_{i_0}S_{\alpha_{i_0}}U^{k_{i_0}}S_\beta^*+\sum_{i=l+1}^{N_\varepsilon}  \gamma_iS_{\alpha_i}U^{k_i}S_{\beta_i}^*+R_\beta$ still satisfies the inequality $\|dP-T_\varepsilon'\|<\varepsilon$. Furthermore, we also have the inequality $|\gamma_{i_0}-d_{\Psi(k)}|<\varepsilon$ per every $k\in I_\beta$. It is now clear that if we repeat this procedure as many times as needed we can get rid of all overlappings. By doing so we end up with a new algebraic 
approximant, which with a very slight abuse of notation we continue to denote by $T_\varepsilon$, given by a sum of the type  $\sum_{i=1}^{N_\epsilon}\gamma_i S_{\alpha_i}U^{k_i}S_{\beta_i}^*$ with $\sum_i S_{\beta_i}S_{\beta_i}^*=
\sum_i S_{\alpha_i}S_{\alpha_i}^*=1$ and $\|dP-(T_\varepsilon+R_\varepsilon)\|<\varepsilon$, where $R_\varepsilon$
is a finite-rank operator. Furthermore, for every $i=1,2,\ldots, N_\epsilon$ there exists a set $I_{\beta_i}\subset A_{\beta_i}$ such that
$A_{\beta_i}\setminus I_{\beta_i}$ is finite and 
\begin{equation}\label{incoeff}
|\gamma_i-d_{\Psi(k)}|<\varepsilon\quad\rm{for\, every}\,k\in I_{\beta_i}.
\end{equation}

Choosing $\varepsilon=\frac{1}{n}$ we get a sequence $\{T_n+R_n: n\in\mathbb{N}\}$, where $T_n$ is an operator of the form
$\sum_i \gamma_i S_{\alpha_i}U^{k_i}S_{\beta_i}^*$ with all the properties pointed out above and $R_n$ is a finite-rank operator, such that $\|v-(T_n+R_n)\|$ goes to zero.
Note that each $T_n$ factors as a product of the form $(\sum_{i=1}^l \gamma_i S_{\alpha_i}S_{\alpha_i}^*)(\sum_{j=1}^l S_{\alpha_j}U^{k_j}S_{\beta_j}^*)$. In 
other terms, each $T_n$ can be seen as the product 
$d_nP_n$, where $d_n$ is a diagonal operator in $\CD_2$ and $P_n$ a permutative operator in $\CQ_2$.
Now thanks to inequality \eqref{incoeff}, the sequance $\{d_n\}$ is immediately seen to be a Cauchy sequence with respect to the uniform norm
in $\mathcal{B}(\ell_2(\mathbb{Z}))$. Therefore, it converges to a certain $d'$, which is a diagonal operator in $\CD_2$.
We then show that the sequence $\{P_n:n\in\mathbb{N}\}$ must stabilize to a certain $P'$. 
We will argue by contradiction. Indeed, let $T_n$ and $T_m$ be such that two corresponding permutative factors $P_n$ and $P_m$ differ. 
There is no 
loss of generality if we further assume that the $\beta$'s appearing both in $P_n$ and in $P_m$ are all of the same length, say $k$.
Because $P_n$ and $P_m$ are different, there must exist at least one $\beta\in W_2^k$ such that the two corresponding monomials do not coincide, i.e.
$S_{\alpha_i}U^{k_i}S_{\beta}^*\neq S_{\alpha'_i}U^{k'_i}S_{\beta_i}^*$, hence the set $I\doteq \{k\in\mathbb{Z}: P_n e_k\neq P_m e_k\}$ is infinite. \\
The inequality $\|v-(T_n+R_n)\|<\frac{1}{n}$ applied to $T_n$ and $T_m$ leads to 
$$\|d_nP_n-d_mP_m +S\|<\frac{1}{n}+\frac{1}{m}$$
where $S$ is a finite-rank operator as well since it is the difference $R_n-R_m$.
In particular, we find that the inequality $\|d_nP_ne_k-d_mP_m e_k +Se_k\|<\frac{1}{n}+\frac{1}{m}$  holds for every $k\in I$.
But, as we next show, this is absurd as soon as $\frac{1}{n}+\frac{1}{m}<1$.
Indeed, in order for the inequalities to hold true, it is necessary that neither $(Se_k, e_{\Psi_n(k)})$ nor
$(Se_k, e_{\Psi_m(k)})$ vanishes. More precisely, their absolute values must be close to $1$. But then $\|Se_k\|$ is greater than $1$ for every $k\in I$. But this is absurd, as the sequence $\{Se_k:k\in I\}$ should in fact converge to zero, since it is the image through a compact operator of a sequence that weakly converges to zero.
Since the sequence $\{T_n\}$ converges in norm as the product of two converging sequences, 
the sequence $\{R_n\}$ must also converge. Let $R$ be its limit. Since $R$ clearly lies in the intersection $\CQ_2\cap\CK(\ell_2(\mathbb{Z}))$, by virtue of
the simplicity of $\CQ_2$ the operator $R$ is zero.
In other words, we have proved that $V=dP$ is nothing but $d'P'$, whence $d=d'\in \CD_2$ and $P=P'\in \CQ_2$. 
\end{proof} 

\begin{remark}
It is worth stressing that the result obtained above is a genuine generalization of the known result on 
$N_{\CD_2}(\CO_2)$, cf. \cite[Lemma 5.4]{Pow}. 
In order to see this, it is enough to show that a unitary of the form $u\doteq\sum_{i=1}^N S_{\alpha_i}U^{k_i}S_{\beta_i}^*$ will lie in the Cuntz algebra $\CO_2$ if and only if $k_i=0$ for every $i=1,2,\ldots, N$.
To this aim, it is convenient to work in the canonical representation. If  there exists $i_0\in\{1,2,\ldots , N\}$ such that $k_{i_0}$ is not zero, then we may safely suppose $k_{i_0}\geq 1$. Now pick the only $k\in\mathbb{Z}$ such that
$S_{\beta_{i_0}}^* e_k=e_{-1}$. Now $k$ is a negative integer such that
$ue_k=S_{\alpha_i}e_{k_{i_0}-1}=e_{n(k)}$ with $n(k)\geq 0$. Therefore, our unitary $u$ cannot be in $\CO_2$, for 
$\CH_{\pm}\subset \ell_2(\mathbb{Z})$ are invariant subspaces under the action of the Cuntz algebra.
\end{remark}  

\begin{remark}
Not only does our result cover the normalizer of $\CD_2$ in the Cuntz algebra $\CO_2$, but it also allows us to recover Putnam's result 
on the normalizer of the former algebra in the Bunce-Deddens algebra $\CQ_2^\mathbb{T}$ \cite[Lemma 5.1]{PUT}.  Indeed, 
for a unitary of the form $\sum_{i=1}^N S_{\alpha_i}U^{k_i}S_{\beta_i}^*$ to lie in the gauge-invariant subalgebra $\CQ_2^\mathbb{T}$
it is necessary to have $|\alpha_i|=|\beta_i|$, for every $i=1,2,\ldots , N$, which
means for every $i$ there exists an integer $l_i\in\mathbb{Z}$, which will depend on $\alpha_i$ and $\beta_i$, such that
$S_{\beta_i}=U^{l_i} S_{\alpha_i}$. But then our unitary takes the form 
$\sum_{i=1}^N S_{\alpha_i}S_{\beta_i}^*U^{2^{|\beta_i|}k_i}=\sum_{i=1}^N S_{\alpha_i}S_{\alpha_i}^*U^{-l_i}U^{2^{|\beta_i|}k_i}$  
which coincides with the formula given in {\it loc.cit.} as

\begin{align*}
\sum_{i=1}^N U^{-{2^{|\beta_i|}k_i}}U^{l_i}S_{\alpha_i}S_{\alpha_i}^*U^{-l_i} U^{2^{|\beta_i|}k_i}=&
\sum_{i=1}^N U^{-2^{|\beta_i|}k_i} S_{\beta_i}S_{\beta_i}^*U^{2^{|\beta_i|}k_i}=\sum_{i=1}^N S_{\beta_i} U^{-k_i}U^{k_i}S_{\beta_i}^*=1
\end{align*}
\end{remark}

\begin{remark}\label{decomp}
Since any unitary of the form $\sum_{i=1}^N S_{\alpha_i}U^{k_i}S_{\beta_i}^*$ can obviously be rewritten as the product $\left(\sum_{i=1}^N S_{\alpha_i}U^{k_i}S_{\alpha_i}^*\right)\left( \sum_{j=1}^N S_{\alpha_j}S_{\beta_j}^*\right)$, the foregoing result can also be stated in a slightly more intriguing way saying that any unitary in the normalizer 
$N_{\CD_2}(\CQ_2)$ decomposes into the product of a unitary in $N_{\CD_2}(\CQ_2^\mathbb{T})$ and a unitary in
$N_{\CD_2}(\CO_2)$. However, such a decomposition will fail to be unique, and one reason is for example that any diagonal elements
$d\in\CD_2$ sits in both $N_{\CD_2}(\CQ_2^\mathbb{T})$ and $N_{\CD_2}(\CO_2)$.
\end{remark}

\section{A diagrammatic description of the extended Thompson group $\CW$}
The group made up of the unitaries in $\CQ_2$ of the form $\sum _{i\in F} S_{\alpha_i}U^{k_i}S_{\beta_i}^*$, where $F$ is a finite set over which the triples $(\alpha_i,\beta_i, k_i)\in W_2\times W_2\times \mathbb{Z}$ run, will be referred to as  the extended Thompson group $\CW$. This section  provides a graphical description of its elements, which is similar to that for the elements of the genuine Thompson groups $F$, $T$, and $V$. \\

An element will be described by a $4$-tuple $(T_+,T_-,\tau,v)$, where $T_\pm$ are trees with the same number of leaves, say $n$, $\tau$ is a permutation of the set $\{1, \ldots , n\}$, and $v$ is a vector in $\IZ^n$.
Let $x=\sum _{i\in F} S_{\alpha_i}U^{k_i}S_{\beta_i}^*$ be an element of $\CW$. The collection of the indices $\alpha_i$ determines a finite subtree of the infinite binary tree of standard dyadic intervals \cite{CFP}.
Indeed, each $\alpha_i$ represents a path in the infinite binary tree of the standard dyadic intervals, from the root to a leaf. If $\alpha_i(k)=1$ then the $k$-th edge of the path is a left edge, whereas $\alpha_i(k)=2$ means that we are taking the right edge.
This yields the first tree $T_+$. In the same way we may get another tree $T_-$ from the indices $\beta_i$. 
As for the permutation, the leaves of $T_+$ and $T_-$ can be indexed from $1$ to $n$ starting from the left. If $\alpha_i$ is the $p$-th leaf of $T_+$ and $\beta_i$ is the $q$-th leaf of $T_-$, then the permutation is defined by $\tau(p):=q$.
The vector $v$ is given by $(k_1, \ldots , k_n)$.  \\

We now describe how to represent such an element graphically. We draw $T_+$ in the upper-half plane and $T_-$ upside-down in the lower-half plane.  We then join the $i$-th leaf of $T_+$ to the $\tau(i)$-th leaf of $T_-$.
Each leaf of the top tree has a charge given by $v(i)$. 
To take an example, below we represent the graphical description of the unitary $S_1^2 U^{k_1} (S_2S_1)^*+S_1S_2U^{k_2}S_1^* + S_2 U^{k_3}(S_2^2)^*$. In this case, the vector $v$ is $(k_1, k_2, k_3)$ and $\tau= (12)$.
$$
\includegraphics[scale=0.35]{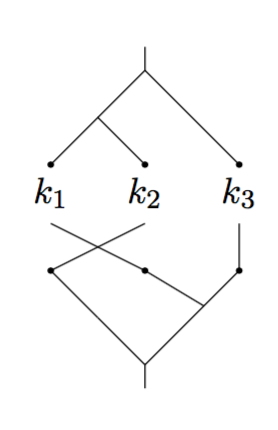}
$$ 

In this pictorial description, there are actually two reduction moves which may be performed and depend on the charge of the leaf.
$$
\includegraphics[scale=0.3]{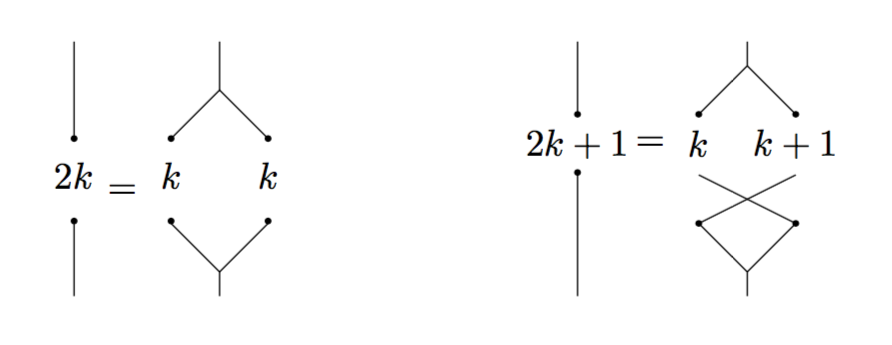}
$$ 
As shown by the following computations,
these reductions correspond to the insertion of the Cuntz relation $S_2S_2^*+S_1S_1^*=1$
\begin{align*}
S_\alpha U^{2k} S_\beta^* & = S_\alpha S_1S_1^*U^{2k} S_\beta^* +S_\alpha S_2S_2^*U^{2k} S_\beta^* = S_\alpha S_1U^{k}S_1^* S_\beta^* +S_\alpha S_2U^{k} S_2^*S_\beta^*\\
& = S_{\alpha 1} U^{k}S_{\beta 1}^*+S_{\alpha 2} U^{k}S_{\beta 2}^*\\
& 
= S_\alpha U^{2k} S_1S_1^* S_\beta^*+S_\alpha U^{2k}S_2S_2^* S_\beta^* 
\\
S_\alpha U^{2k+1} S_\beta^* & 
= S_\alpha S_1S_1^*U^{2k+1} S_\beta^* +S_\alpha S_2S_2^*U^{2k+1} S_\beta^*= S_\alpha S_1U^{k}S_1^*U S_\beta^* +S_\alpha S_2U^{k} S_2^*US_\beta^*\\
& 
= S_{\alpha 1} U^{k}S_{\beta 2}^*+S_{\alpha 2} U^{k+1}S_{\beta 1}^*\\
 & 
 = S_\alpha U^{2k+1} S_1S_1^*S_\beta^* +S_\alpha U^{2k+1}S_2S_2^* S_\beta^* 
\end{align*}
Given two elements $(T_+,T,\tau,v)$ and $(T,T_-,\tau',v')$, their product as elements of $\CQ_2$ is given by $(T_+,T,\tau,v)\cdot (T,T_-,\tau',v'):= (T_+,T_-,\tau\circ\tau',v+v')$. Thanks to the reduction moves, this actually describes the multiplication on the whole $\CW$. Clearly, the inverse of an element is given by $(T_+,T_-,\tau,v)^{-1}=(T_-,T_+,\tau^{-1},-v)$.

\begin{remark}
We should also mention that the group $\CW$  has  appeared before in the literature, albeit in different contexts.
For instance, in \cite{NEK}  it is shown how to associate a group  $V_d(G)$ to any given a self-similar action of a group $G$ over an alphabet $X$ of finite cardinality $d$. The groups obtained in this fashion are actually a generalization of the  Higman-Thompson group, and our group $\CW$ corresponds to the case $X=\{1,2\}$ and $G=\IZ$ thought of as the powers of the so-called odometer. 
Unlike $V$, the group $\CW$ is not simple as its abelianization is $\IZ$
\cite[Example 9.16]{NEK}.
\end{remark}

\section{Other normalizers}\label{normalizers}

A number of results about $\Aut_{\CD_2}(\CQ_2)$ are contained in \cite{ACR2}.
At present the structure of the group $\Aut(\CQ_2,\CO_2)$ is for the most part unknown, although we do know some remarkable examples, namely the extended
flip-flop automorphism and the gauge automorphisms. 
A related question is whether there exist automorphisms of $\CQ_2$ that restrict to proper endomorphisms of $\CO_2$.
At any rate, most of the subsequent discussion is concerned only with inner automorphisms of $\CQ_2$.

\subsection{Unitaries in the Bunce-Deddens algebra normalizing $\CO_2$}

If $\alpha$ is an automorphism of $\CQ_2$ that leaves $\CO_2$ globally invariant, the natural question immediately arises whether 
the restriction $\alpha\upharpoonright_{\CO_2}$ is an automorphism of the Cuntz algebra as well. Apart from the trivial situation where our automorphism
is of finite order, in which case its restriction is immediately seen to be an automorphism of $\CO_2$, a complete answer has not been given, not even in the simpler yet interesting case where $\alpha$ is an inner automorphism of $\CQ_2$. 
Now the question is recast by asking whether $w\CO_2w^*\subset \CO_2$, for a given unitary $w$ in $\CQ_2$,  can only hold true if $w$ lies in
$\CO_2$. In its full generality the latter question is still unexpectedly hard to answer. Therefore, we ought to start  with $w$ of a particular
form instead. Assuming $w\in\CQ_2^{\IT}$ seems to be a good work hypothesis to begin our discussion with. 

Given any infinite multi-index $\alpha\in\{1,2\}^\IN$, we denote by $\alpha(k)$ the multi-index of length $k$ that 
is obtained out of $\alpha$ by taking its first $k$ entries, i.e. $\alpha(k)\doteq (\alpha_1, \alpha_2, \ldots, \alpha_k)$.  

We recall from \cite[Section 7]{ACR2} the following identity 
$$S_i^*dS_i (x)= d(ix) \ , \quad d\in\CD_2, \;  x\in K=\{1,2\}^\IN\; . $$

\begin{lemma}\label{algebraic}
If $w=\sum_{i\in F} d_i U^i$, where $F\subset \IZ$ is a finite subset, is different from zero, then there exist 
$\alpha\in \{1,2\}^\IN$ and $h\in F$ such that 
$$
\lim_k S_{\alpha(k)}^* w U^{-h} S_{\alpha(k)}=d_h(\alpha)1
$$
with $d_h(\alpha)$ different from zero as well.
\end{lemma}
\begin{proof}
By definition if $w$ is different from zero, there must exist an $h\in F$ such that $d_h \in \CD_2$ is not zero, that is
$d_h(\alpha)\neq 0$ for some $\alpha \in K=\{1,2\}^{\IN}$.
We now prove that the limit holds with $\alpha$ and $h$ chosen as above. To this aim, 
it is enough to show $\lim_k S_{\alpha(k)}^*d_hS_{\alpha(k)}= d_h(\alpha)1$
and $\lim_k S_{\alpha(k)}^*d_i U^{i-h}S_{\alpha(k)}= 0$, for every $i\neq h$.
The first is easily proved as the equality $S_{\alpha(k)}^* d_h S_{\alpha(k)}(x)= d_h(\alpha(k)x)$
shows that $S_{\alpha(k)}^* d_h S_{\alpha(k)}(x)$ converges to $d_h(\alpha)1$ pointwise. But on the other hand,
by a straightforward adaptation of a result in \cite{ACR} we also know that the sequence  actually converges in norm. 
The second limit is in fact a consequence of a more general fact, namely that for any $d\in\CD_2$  and 
$l\neq 0$ we have $\lim_k S_{\alpha(k)}^*dU^l S_{\alpha(k)}=0$. This is in turn proved as follows.
Pick a sequence $\{x_j:j\in\IN\}$ such that $x_j\in\CD_2^j$ and $\|x_j- d\|\rightarrow 0$, then 
$$
\| S_{\alpha(k)}^* dU^{l}S_{\alpha(k)}\|\leq \| S_{\alpha(k)}^* (d-x_j)U^{l}S_{\alpha(k)}\|+\| S_{\alpha(k)}^* x_jU^{l}S_{\alpha(k)}\| \leq \| d-x_j\|
$$
as soon as $k\geq j+1$ and $2^k> |l|$, because the second term vanishes if $2^k> |l|$.
\end{proof}
Our next goal is to extend the reach of the foregoing lemma to cover the whole $\CQ_2^{\IT}$.
With this in mind, we recall that this algebra can be more conveniently thought of as a crossed product given by the action
of $\IZ$ on the diagonal $\CD_2$ through the odometer map, cf. \cite{ACR2}. In other terms, the map
$\Psi:  \CD_2\rtimes \IZ\to\CQ_2^{\IT}$ given by
$\Psi(V)= U$ $\Psi(d)=d$ for all $d\in\CD_2$, extends to an isomorphism,
where $\CD_2\rtimes \IZ$ is understood as $C^*(\CD_2, V)$, with $V$ being a unitary whose adjoint action coincides with the
odometer, and $\CQ_2^\IT$ is understood as $C^*(\CD_2, U)$.
This makes it plain that $\beta_z(U):=zU$ and $\beta_z(d):=d$, $z\in\IT$, define a group of automorphism of $\CQ_2^\IT$.
These enable us to define the maps $\Phi_i: \CQ_2^\IT\to \CD_2$ by setting $\Phi_i(x):=\int_\IT \beta_z(xU^{-i})dz$, cf. \cite[p. 223]{Davidson},
which can be regarded as generalized Fourier coefficients. Indeed, it is no coincidence that a version of Fej\'er's theorem holds for $\CQ_2^\IT$ as well.
\footnote{ Recently, a Fej\'er-type theorem has also been proved \cite{ACR5} for quite an ample class of
$C^*$-algebras, which includes the $2$-adic ring $C^*$-algebra.}
More precisely, if we define  $\Sigma_n(x)\doteq \sum_{j=-n}^n \left(1-\frac{|j|}{n+1}\right)\Phi_j(x)U^j$, then
$\Sigma_n(x)$ can be shown to converge to $x$ in norm for every $x\in\CQ_2^{\IT}$, see e.g. \cite[Theorem VIII.2.2, p. 223]{Davidson}
\begin{remark}
For any unitary $d\in\CD_2$, one can define the automorphism of the Bunce-Deddens $\beta_d$ given by $\beta_d(U):=dU$, $\beta_d(\tilde d):=\tilde{d}$ for all $\tilde{d}\in \CD_2$. 
\end{remark}

\begin{lemma}\label{nonzerolimit}
For any non-zero $w\in\CQ_2^{\IT}$ there exist $h\in \IZ$ and $\alpha\in\{1,2\}^\IN$ such that
$$
\lim_k S_{\alpha(k)}^*w U^{-h} S_{\alpha(k)}= \Phi_h(w)(\alpha)1
$$
and $\Phi_h(w)(\alpha)$ is different from zero as well.
\end{lemma}
\begin{proof}
Since $w$ is different from zero, there exists an $h\in\IZ$ such that
$\Phi_h(w)$ is not zero, which means there is $\alpha\in\{1,2\}^\IN$ such that
$\Phi_h(w)(\alpha)\neq 0$. Now $S_{\alpha(k)}^*w U^{-h} S_{\alpha(k)}$ tends to $\Phi_h(w)(\alpha)$ as
a straightforward  application of the aforementioned Fej\'er theorem. Indeed, we have
\begin{align*}
& \| S_{\alpha(k)}^*w U^{-h} S_{\alpha(k)} -\Phi_h(w)(\alpha)1\|  \leq \| S_{\alpha(k)}^*(w-\Sigma_n(w)) U^{-h} S_{\alpha(k)}\|  \\
& \qquad + \| S_{\alpha(k)}^*\Sigma_n(w) U^{-h} S_{\alpha(k)}-\Phi_h(\Sigma_n(w))(\alpha)1\|\\
& \qquad +\|\Phi_h(\Sigma_n(w))(\alpha)1-\Phi_h(w)(\alpha)1\|\\
& \qquad \leq  2\|\Sigma_n(w)-w\|+\| S_{\alpha(k)}^*\Sigma_n(w) U^{-h} S_{\alpha(k)}-\Phi_h(\Sigma_n(w))(\alpha)1\|
\end{align*}
Given $\varepsilon>0$ there is $N_\varepsilon\in\IN$ such that $\|\Sigma_n(w)-w\|<\frac{\varepsilon}{3}$ for $n \geq N_\varepsilon$. Now apply
Lemma \ref{algebraic} to $\Sigma_n(w)$, where $n$ is any fixed integer greater than $N_\varepsilon$ and $|h|$, to get a $K_\varepsilon\in\IN$ such that
$k\geq K_\varepsilon$ implies $\| S_{\alpha(k)}^*\Sigma_n(w) U^{-h} S_{\alpha(k)}-\Phi_h(\Sigma_n(w))(\alpha)1\|<\frac{\varepsilon}{3}$.
\end{proof}

\begin{thm}\label{normalO2}
If $w$ is a unitary in $\CQ_2^{\IT}$ such that $w \CO_2w^*\subset\CO_2$  
 then $w\in\CF_2$ and the inclusion is actually a set equality.
\end{thm}

\begin{proof}
Thanks to Lemma \ref{nonzerolimit} there exist $h\in\IZ$ and an infinite multi-index $\alpha\in\{1,2\}^\IN$
 such that $S_{\alpha(k)}^*w U^{-h} S_{\alpha(k)}$ converges to $\lambda 1$ with 
$\lambda\neq 0$ when $k$ goes to infinity.\\
Since $\Ad(w)\upharpoonright_{\CO_2}=\lambda_{w\varphi(w)^*}$, the unitary $w\varphi(w)^*$ is an element of $\CF_2$ and such is $w\varphi^k(w^*)$, for every $k\geq 1$.
Now there exist $l=l(h,\alpha)\in\IZ$ and an infinite multi-index $\beta$ such that 
$U^{-h}S_{\alpha(k)}=S_{\beta(k)}U^{-l}$, i.e. 
$U^hS_{\beta(k)}=S_{\alpha(k)}U^l$, 
for $k$ sufficiently large. Since we have
\begin{align*}
\CF_2\ni S_{\alpha(k)}^* w \varphi^k(w^*)S_{\beta(k)} & = S_{\alpha(k)}^* wU^{-h}U^h \varphi^k(w^*)S_{\beta(k)}=S_{\alpha(k)}^* wU^{-h}U^h S_{\beta(k)}w^*\\
& = S_{\alpha(k)}^*wU^{-h}S_{\alpha(k)}U^l w^*\underset{k\rightarrow\infty} {\rightarrow} \lambda U^l w^*
\end{align*}
we see that $U^lw^*=w_0^*\in \CF_2$. In other terms, $w=w_0U^{l}$. The proof will be completed as soon as we show that
$l$ must be zero. But from the inclusion $w_0U^{l}\CO_2 U^{-l}w_0^*\subset \CO_2$ we get
$U^{l}\CO_2U^{-l}\subset w_0^*\CO_2w_0=\CO_2$, which is possible only if $l=0$, for
$U^{l}S_1U^{-l}$ is never in $\CO_2$ unless $l=0$.
\end{proof}

In particular, with $u\doteq w\varphi(w)^*$, one has $\lambda_u(\CF_2)=\CF_2$ and $\lim_k \varphi^k(w)w^*\varphi(w)\varphi^k(w^*)=w^*\varphi(w)\in\CF_2$ in norm, also cf.
Theorem \ref{characterization}.

\subsection{Some consequences}

We can also derive the following results.
\begin{corollary}
The intersection $N_{\CD_2}(\CQ_2)\cap N_{\CO_2}(\CQ_2)$ reduces to $N_{\CD_2}(\CO_2) (= \CU(\CD_2) \cdot \CS_2)$.
\end{corollary}
\begin{proof}
As the inclusion  $N_{\CD_2}(\CO_2)\subset N_{\CD_2}(\CQ_2)\cap N_{\CO_2}(\CQ_2)$ is trivial, we only need to prove that any unitary $u$ in the intersection  $N_{\CD_2}(\CQ_2)\cap N_{\CO_2}(\CQ_2)$ actually lies in the Cuntz algebra $\CO_2$.
Since our unitary $u$ sits in particular in $N_{\CD_2}(\CQ_2)$, it decomposes as $u=u_1u_2$ with
$u_1\in N_{\CD_2}(\CQ_2^\mathbb{T})$ and $u_2\in N_{\CD_2}(\CO_2)$ by virtue of Remark \ref{decomp}. 
But because $u$ also normalizes $\CO_2$, we find that $u_1(u_2xu_2^*)u_1^*$ continues to be in $\CO_2$ if $x$ is in
$\CO_2$. In other terms, $u_1$ is an element of the Bunce-Deddens algebra such that $u_1\CO_2u_1^*\subset\CO_2$.
This being the case,  Theorem \ref{normalO2} applies showing that $u_1$ lies in fact in $\CF_2$. But then $u$ must  be in
$\CO_2$ too as the product of two unitaries both  lying in $\CO_2$.
\end{proof}

\begin{corollary} \label{normD2}
We have
$$ N_{\CD_2}(\CQ_2^\IT) \cap N_{\CO_2}(\CQ_2) = N_{\CD_2}(\CF_2) (= \CU(\CD_2) \cdot \CP_2) \ . $$
\end{corollary}

\begin{remark}
Among other things, the former results provide more information about
$N_{\CO_2}(\CQ_2)$. Notably, it seems to vaguely support our guess that the normalizer $N_{\CO_2}(\CQ_2)$
may be exhausted by $\mathcal{U}(\CO_2)$. At the very least, it certainly settles those unitaries in $N_{\CO_2}(\CQ_2)$
that also leave $\CD_2$ globally invariant.
\end{remark}

The next result gives us an explicit description of the normalizer
$N_{\CQ_2^\IT}(\CQ_2)$, which by definition is the set of those unitaries $w\in\CQ_2$ such that
$w Q_2^\IT w^*=\CQ_2^\IT$.

\begin{proposition}\label{normalizer}
The normalizer $N_{\CQ_2^\IT}(\CQ_2)$ coincides with $\CU(\CQ_2^\IT)$.
\end{proposition}

\begin{proof}
As the inclusion $\CU(\CQ_2^\IT)\subset N_{\CQ_2^\IT}(\CQ_2)$ is obvious, we only need to prove the reverse inclusion.
If $w\in\CQ_2$ is in $N_{\CQ_2^\IT}(\CQ_2)$, then $wxw^*\in\CQ_2^\IT$ for every $x\in\CQ_2^\IT$, which means
$\tilde\alpha_t(wxw^*)=wxw^*$ for every $t\in\IR$, where $\tilde\alpha_t$ is the gauge automorphism
of $\CQ_2$ corresponding to $2^{it}$. In particular, we find that $w^*\tilde\alpha_t(w)x=xw^*\tilde\alpha_t(w)$ for every $x\in\CQ_2^\IT$, whence
$w^*\tilde\alpha_t(w)=\chi(t)1$ as the relative commutant $(\CQ_2^\IT)' \cap \CQ_2$ is trivial, see \cite[Corollary 3.13]{ACR}.
In other terms, we get that the equality $\alpha_t(w)=\chi(t)w$  holds for every $t\in\IR$, where $\chi$ is immediately
seen to be a character of $\IR$, i.e. $\chi(t)=e^{ikt}$ for some $k\in\IZ$. The proof will then be completed once $k$ is shown to be
zero. This goal can in turn be accomplished by using the $\beta$-KMS state $\omega$ associated with $\{\tilde\alpha_t: t\in\IR\}$, where
$\beta$ is actually $1$, see \cite[Prop. 4.2]{CuntzQ}.  
Indeed, on the one hand we have $\omega(w\tilde\alpha_t(w^*))=e^{-ikt}$, but on the other hand 
$\omega(w\tilde\alpha_t(w^*))=\omega(\tilde\alpha_{t+i\beta}(w^*)w)=e^{k\beta} e^{-ikt}$, which
forces $e^{k\beta}$ to equal $1$, i.e. $k=0$ as $\beta=1\neq 0$.
\end{proof}

At this point Theorem \ref{normalO2} can be reformulated in a slightly more intrinsic way. 
\begin{corollary}\label{inner}
Let $\alpha$ be an inner automorphism of $\CQ_2$ such that $\alpha(\CQ_2^\IT)=\CQ_2^\IT$  and $\alpha(\CO_2)\subset\CO_2$, then
$\alpha$ is the canonical extension of an inner automorphism of $\CF_2$.
\end{corollary}
\begin{proof}
A straightforward application of Proposition \ref{normalizer}.
\end{proof}

It seems of interest also to determine whether the only unitaries in $\CQ_2^{\IT}$ normalizing $\CF_2$ (or even mapping $\CF_2$ into $\CF_2$) are those in $\CF_2$.
Unfortunately, we do no to have an answer yet. At any rate, we do have the following result, whose proof is omitted as 
it can be done in exactly the same way as in  Proposition \ref{normalizer} by taking account of the equality $\CF_2'\cap\CQ_2=\IC1$, see \cite[Theorem 3.11]{ACR}.

\begin{proposition}\label{normF2}
The normalizer  $N_{\CF_2}(\CQ_2)$ is equal to $N_{\CF_2}(\CQ_2^\IT)$.  
\end{proposition}

We now state a result that describes the intersection $N_{\CD_2}(\CQ_2) \cap N_{\CF_2}(\CQ_2) \cap N_{\CO_2}(\CQ_2)$, for
which we need not know what the three normalizers separately are.

\begin{corollary} One has
$$N_{\CD_2}(\CQ_2) \cap N_{\CF_2}(\CQ_2) \cap N_{\CO_2}(\CQ_2) = N_{\CD_2}(\CF_2)=N_{\CD_2}(\CO_2)\cap N_{\CF_2}(\CO_2)$$
\end{corollary}

\begin{proof}
We only prove the first equality, as the second is basically known and is easy to check.
The inclusion $N_{\CD_2}(\CF_2)\subset N_{\CD_2}(\CQ_2) \cap N_{\CF_2}(\CQ_2) \cap N_{\CO_2}(\CQ_2)$ is obvious.
For the other inclusion, note that any $x\in N_{\CD_2}(\CQ_2) \cap N_{\CF_2}(\CQ_2) \cap N_{\CO_2}(\CQ_2)$  lies in
$\CQ_2^{\mathbb{T}}$ by  Proposition \ref{normF2}. The conclusion is then reached straightforwardly by an application of
Corollary \ref{normD2}.
 \end{proof}

\begin{proposition}
The intersection $N_{\CF_2}(\CQ_2)\cap N_{\CO_2}(\CQ_2)$ reduces to $\CU(\CF_2)$.
\end{proposition}

\begin{proof}
By Proposition \ref{normF2} the normalizer of $\CF_2$ in $\CQ_2$ coincides with $N_{\CF_2}(\CQ_2^\IT)$. Therefore, the intersection
$N_{\CF_2}(\CQ_2)\cap N_{\CO_2}(\CQ_2)$ is actually given by $N_{\CF_2}(\CQ_2^{\IT})\cap N_{\CO_2}(\CQ_2)$. All unitaries in the
UHF algebra $\CF_2$ are obviously contained in the above inclusion. On the other hand, Theorem \ref{normalO2} says the intersection must be contained in
$\CU(\CF_2)$ as well, which ends the proof.
\end{proof}

Theorem \ref{normalO2} and Corollary \ref{inner} could be improved in some respects.
For instance, one may  also want to consider more general (possibly outer) automorphisms of $\CQ_2$ leaving $\CO_2$ globally invariant. 
This may in fact come in useful to deal with permutative endomorphisms.
Furthermore, one may even ask whether the normalizer of $\CO_2$ in $\CQ_2$ is $\CU(\CO_2)$.

\begin{remark}
The inclusion $\CO_2\subset \CO_2 \rtimes_{\lambda_f} \IZ_2\cong \CO_2$, where $\lambda_f$ 
 is the flip-flop automorphism that switches the two generating isometries, provides an example where the normalizer
of the Cuntz algebra, thought of as a subalgebra of a bigger algebra, which in this case is again the Cuntz algebra
up to an isomorphism, is not exhausted by the unitaries of $\CO_2$. Indeed, the unitary 
$w\in\CO_2\rtimes_{\lambda_f}\IZ_2$ that implements the action of the flip-flop on $\CO_2$ is certainly
not contained in the Cuntz algebra since $\lambda_f$ is outer. 
It is also worth noting that the relative commutant  $\CO_2'\cap (\CO_2 \rtimes_{\lambda_f} \IZ_2)$ is trivial,
cf. \cite[Remark 5.9]{Choi}
\end{remark}

\subsection{A class of automorphisms of $\CO_2$}

In passing, we would like to take this opportunity to point out a result that yields a complete description of those unitaries 
in $u$ in the UHF algebra $\CF_2$ whose corresponding endomorphism $\lambda_u$ is actually an automorphism, which is much in the spirit of \cite[Theorem 3.2]{CoSz11}. 
Its main interest has admittedly little to do with the scope of the present paper. Nevertheless, we feel it deserves to be included all the same 
because it might be further developed in future work as well as framing Theorem \ref{normalO2} in a more general picture.
To this end, the following straightforward remark, which is general in character, is vital.

\begin{remark}\label{gen}
If $\lambda_u$ is an automorphism and $\lambda_u^{-1}=\lambda_v$, then $\lambda_u(v)=u^*$ and $\lambda_v(u)=v^*$.
Indeed, by definition we have $\lambda_u\circ\lambda_v=\lambda_v\circ\lambda_u=\textrm{id}_{\CO_2}=\lambda_{1}$.
Now $\lambda_u\circ\lambda_v=\lambda_{\lambda_u(v)u}$ and $\lambda_v\circ\lambda_u=\lambda_{\lambda_v(u)v}$, hence
$\lambda_u(v)u=1$ and $\lambda_v(u)v=1$, i.e. $\lambda_u(v)=u^*$ and $\lambda_v(u)=v^*$, as maintained.
In particular, it follows that if $\lambda_u\in\Aut(\CO_2)$ then $\lim_k \Ad(u_k)(v)=\lambda_u(v)=u^*$.
\end{remark}

\begin{thm}\label{characterization}
Given a unitary $u$ in $\CF_2$, set $u_k\doteq u\varphi(u)\varphi(u^2)\ldots\varphi^{k-1}(u)$, for any  $k\in\IN\setminus\{0\}$.
Then  the following are equivalent:
\begin{enumerate}
\item The endomorphism $\lambda_u$ is an automorphism of $\CO_2$.
\item There exists a unitary $v\in\CF_2$ such that $\lim_k \Ad (u_k) v=u^*$.
\item The sequence $\{u_k^*u^*u_k: k\in\IN\}$ is norm convergent.
\end{enumerate}
Moreover, as soon as any of the three conditions above is satisfied, the $v$ in (2) is the limit of the sequence in (3) and $\lambda_u^{-1}=\lambda_v$.
\end{thm} 

\begin{proof}
The implication $(1)\Rightarrow (2)$ is nothing but the content of Remark \ref{gen}.
On the other hand, if there exists a unitary $v\in\CF_2$ such that $\lim_k\Ad u_k(v)=u^*$, then 
$\lambda_u(v)=u^*$, hence $\lambda_u$ is surjective. In other terms, the implication
$(2)\Rightarrow (1)$  holds too.
We now prove $(1)\Rightarrow (3)$. If $\lambda_u$ is an automorphism, then by (2) there exists a unitary $v\in \CF_2$
such that $\|u_kvu_k^*- u^*\|$ tends to zero. By the very definition of the UHF subalgebra $\CF_2$ we can assume
that $v=\lim v_i$ with $v_i\in \CF_2^i$ fo every $i\in\IN$. By continuity, we get the equalities
$$\lambda_u(v)=\lim_i\lambda_u(v_i)=\lim_i u_iv_iu_i^*
$$
namely  $\| \lambda_u(v)-u_iv_iu_i^*\|$ goes to zero. But then
$$
\|u_i^*u^*u_i-v\|\leq \| u_i^*u^*u_i-v_i \|+\|v_i -v\|\rightarrow 0
$$
For $(3)\Rightarrow (1)$, let us define $v\doteq \lim_k u_k^*uu_k \in \CU(\CO_2)$. We then want to show that
$\lambda_u(v)=u^*$, whence $\lambda_u$ is surjective. 
Now $\lambda_u(v)=\lim_j u_jvu_j^*=\lim_j u_j(\lim_i u_i^*uu_i)u_j^*$. In other terms, 
$\|u_j^*\lambda_u(v)u_j-\lim_i u_i^*uu_i\|\rightarrow 0$, that is for any given 
$\varepsilon>0$ there exists $n_\varepsilon\in\IN$ such that $j\geq n_\varepsilon$
implies  $\| u_j^*\lambda_u(v)u_j-\lim_iu_i^*uu_i \|<\frac{\varepsilon}{2}$.

$$
\|u_j^*\lambda_u(v)u_j-u_i^*uu_i\|\leq \|u_j^*\lambda_u(v)u_j -\lim_i u_i^*uu_i \|+\|\lim_i u_i^*uu_i-u_i^*uu_i\|<\varepsilon
$$

for every $i,j\geq N_\varepsilon\doteq {\rm max}\{n_\varepsilon,m_\varepsilon\}$, where $m_\varepsilon$ is any integer such that $i\geq m_\varepsilon$ implies
$\|\lim_i u_i^*uu_i-u_i^*uu_i\|<\frac{\varepsilon}{2}$. In particular, if we choose $i=j\geq N_\varepsilon$, we find 
$\|\lambda_u(v)-u^*\|=\|u_i^*(\lambda_u(v)-u^*)u_i\|\leq\varepsilon$, hence $\lambda_u(v)=u^*$.
\end{proof}

\begin{remark}
Let $u$ be a unitary in $\CF_2$ such that $\lambda_u$ is an automorphism. Since $\lambda_u^{-1}=\lambda_v$, with $v$ that is still
in $\CF_2$, we see that $\lambda_u$ is actually an element of $\Aut(\CO_2, \CF_2)\doteq\{\alpha\in\Aut(\CO_2): \alpha(\CF_2)=\CF_2\}$ 
\end{remark}

\section{The inclusion of $\CO_2 \subset \CQ_2$ is not regular}

The present section collects some results on the unitary normalizer of $\CO_2$ in $\CQ_2$, namely  the group $N_{\CO_2}(\CQ_2)\doteq\{v \in \CU(\CQ_2) \ | \ v \CO_2 v^* = \CO_2\} \subset \CU(\CQ_2)$.
Needless to say, this normalizer is $\Aut_{\CO_2}(\CQ_2)$-invariant, in particular invariant under the action of the extended gauge and flip-flop automorphisms.\\

We start our discussion with a technical lemma, which roughly says that no unitary operator of $\CQ_2\subset\CB(\ell_2(\IZ))$ can map $\CH_+$ to $\CH_-$ and $\CH_-$ to $\CH_+$, where $\CH_+$ and $\CH_-$ are the closed subspaces of $\ell_2(\IZ)$ given by
$\overline{{\rm span}}\{e_k:\,\,k\geq 0\}$ and $\overline{{\rm span}}\{e_k:\,\,k<0\}$ respectively.

\begin{lemma}\label{noexchange}
There is no unitary $u$ in $\CQ_2$ such that $u\mathcal{H}_\pm=\mathcal{H}_\mp$.
\end{lemma}

\begin{proof}
We shall argue by contradiction. Let $u$ be a unitary in $\CQ_2$ such that $u\mathcal{H}_+=\mathcal{H}_-$.
Let $\epsilon>0$ be any real number strictly less than $1$.
Since $u$ belongs to $\CQ_2$ then there exists an element of the form $\sum_{i=1}^N c_i S_{\alpha_i}S_{\beta_i} U^{k_i}$, $c_i\in\IC$ and 
$\alpha_i,\beta_i\in W_2$, such that $\|u- \sum_{i=1}^N c_i S_{\alpha_i}S_{\beta_i} U^{k_i} \|< \varepsilon$. Let now $M$ be $\max\{|k_i|:i=1, 2, \ldots , N\}$.
Then $\| ue_M-\sum_{i=1}^N c_i S_{\alpha_i}S_{\beta_i} U^{k_i}e_M \|= \|ue_M- \sum_{i=1}^N c_i S_{\alpha_i}S_{\beta_i} e_{k_i+M}\|<\varepsilon$.
Since $ue_M$ lies in $\mathcal{H}_-$, it is orthogonal to $\sum_{i=1}^N c_i S_{\alpha_i}S_{\beta_i} e_{k_i+M}\in\mathcal{H}_+$, but then 
$\|ue_M\|<\varepsilon<1$, which is clearly absurd. 
\end{proof}

\begin{remark}
As a consequence, the only operator in $\CQ_2$ that sends $\CH_{\pm}$ to $\CH_{\mp}$ is the null operator. 
\end{remark}

The above lemma is instrumental in proving the next result. Although not still the complete characterization of $N_{\CO_2}(\CQ_2)$,  it  does have the merit of limiting the sought normalizer.  More precisely, the result says that $N_{\CO_2}(\CQ_2)$ cannot be larger than $\CU(\CQ_2\cap O_2'')$.   

\begin{proposition}\label{inclusion}
The normalizer $N_{\CO_2}(\CQ_2)$ is contained in $\CO_2'' \cap \CQ_2$.
\end{proposition}

\begin{proof}
Let $u\in\CQ_2\subset\mathcal{B}(\ell_2(\IZ))$ be a unitary that leaves $\CO_2$ globally invariant, that is $u\CO_2u^*=\CO_2$. Then it also 
leaves $\CO_2'$ invariant. But $\CO_2'=\IC E_+\oplus \IC E_-$,  where $E_\pm$ is the orthogonal projection onto $\mathcal{H}_\pm$, see \cite[Section 2] {ACR}. 
Accordingly, there are only two cases that can occur. Either $uE_\pm u^*= E_\pm$ or $u E_\pm u^*=E_\mp$.  However, Lemma \ref{noexchange} says that the second situation will not occur. But then
$uE_\pm=E_\pm u$, which means $u$ is in $\CO_2''$, as we wanted to prove.
\end{proof}

Notably, the result also enables us to see that the inclusion of $\CO_2$ in $\CQ_2$ is not regular, to wit the $C^*$-algebra generated by its normalizer fails to be the whole $\CQ_2$.

\begin{thm}
The $C^*$-subalgebra generated by $N_{\CO_2}(\CQ_2)$ is properly contained in $\CQ_2$.
\end{thm}

\begin{proof}
Suppose on the contrary this subalgebra does exhaust $\CQ_2$. The intersection $\CO_2''\cap\CQ_2$ should then coincide
with $\CQ_2$, because of the inclusion $N_{\CO_2}(\CQ_2)\subset \CO_2''\cap \CQ_2$ proved in Proposition \ref{inclusion}.
However, the generator $U$ does not sit in $\CO_2''\cap \CQ_2$ since it does not leave $\mathcal{H}_-$ invariant, cf. \cite[Section 2]{ACR}.
\end{proof}
The $C^*$-algebra $C^*(N_{\CO_2}(\CQ_2))$, which is intermediate between $\CO_2$ and $\CQ_2$, is obviously invariant under the extended gauge and flip-flop automorphisms.
As a matter of fact, we would be inclined to believe that the inclusion of $\CO_2$ in $\CQ_2$ is singular, that is the normalizer
$N_{\CO_2}(\CQ_2)$ should reduce to $\mathcal{U}(\CO_2)$. However,  for the time being all we can do is state a partial  result that nonetheless seems to support 
our guess. What we prove is the Cuntz algebra $\CO_2$ is never invariant under the action of a one-parameter group of unitaries
$u_t=e^{ita}$, where $a$ is a self-adjoint element of $\CQ_2$, unless $a$ is already contained in $\CO_2$.
\begin{proposition}
If $a=a^*$ is a self-adjoint element in $\CQ_2$ such that $e^{ita}\CO_2 e^{-ita}= \CO_2$ for any $t\in\IR$, then $a$ sits in the Cuntz
algebra $\CO_2$.
\end{proposition}

\begin{proof}
The condition that $\CO_2$ is invariant under the one-parameter group $u_t\doteq e^{ita} $ generated by $a$ 
says that the commutator  $[x,a]$ is in $\CO_2$ for any $x\in\CO_2$. In other terms $\CO_2 \ni x\mapsto [x, a]\in\CO_2$ is a bounded
derivation. Since $\CO_2$ is simple, the derivation must be inner by virtue of a classical result by Sakai \cite{Sakai}.
In other words,  there exists $b\in\CO_2$ such that
$[x,a]=[x,b]$ for any $x\in \CO_2$. Therefore, the difference $a-b$ lies in the relative commutant $\CO_2'\cap\CQ_2$. Since the latter is trivial, see
\cite[Section 3.2]{ACR}, we find $a=b+\lambda1$, for some $\lambda\in\IC$. In particular, $a$ is an element in $\CO_2$.
\end{proof}


Before leaving the section, some comments intended as an outlook for the foreseeable future are in order.
One way to prove that $N_{\CO_2}(\CQ_2)$ does in fact
coincide with $\CU(\CQ_2)$ could be to show that the intersection $\CO_2''\cap\CQ_2$ is just $\CO_2$.
More concretely, this amounts to asking if any operator in $\CQ_2\subset\CB(\ell_2(\IZ))$ that also leaves both
$\CH_+$ and $\CH_-$ invariant must lie in $\CO_2$.
This property might be in turn a consequence of an even stronger property which would be worthy of further investigation, namely whether any intermediate subalgebra $\CO_2\subset \CA\subset\CQ_2$ (possibly with some additional properties, e.g. requiring it is gauge-invariant) is trivial, that is $\CA$ is either $\CO_2$ or $\CQ_2$. We plan to go back to this problem elsewhere, not least because
we would like to keep the present work at a reasonable length.\\

\section{Permutative endomorphisms of $\CQ_2$ at an arbitrary level}\label{permutative}

This section aims to present a method by means of which  explicit examples can be exhibited  of permutative endomorphisms of the Cuntz algebra $\CO_2$  associated with unitaries in $\CP_2^k\doteq\{u\in\CU(\CO_2)\; |\; u=\sum_{i=1}^nS_{\alpha_i}S_{\beta_i}^*, |\alpha_i|=|\beta_i|=k \textrm{ for } 1\leq i\leq n, n\in\IN \}$ that extend to  $\CQ_2$ for any $k\in\IN$.
In addition, the technique employed also provides a lower bound for the number of such unitaries, which shows it must grow at least as fast as $(2^k!)^2$.

\subsection{Case $\CP_2^2$}

Let $u \in \CP_2^2 \subset \CF_2^2$ be a permutation matrix, and let $\lambda_u$ be the associated permutative endomorphism of $\CO_2$. 
Recall that, in general, $\lambda_u \in {\rm End}(\CO_2)$ extends to $\CQ_2$ if and only if, setting $\tilde{S}_1 = u S_1$ and $\tilde{S}_2 = u S_2$, there exists $\tilde{U} \in \CQ_2$ such that 
\begin{equation}\label{ext1}
\tilde{U} \tilde{S}_2 = \tilde{S}_1 
\end{equation}
\begin{equation}\label{ext2}
\tilde{U} \tilde{S}_1 = \tilde{S}_2 \tilde{U} 
\end{equation}

In \cite{ACR3} it was found that there are exactly ten extendible permutative endomorphisms rising from unitaries in $P_2^2$.
Among other things, their unique extensions all continue to commute with the (extended) gauge automorphisms. 
In this list there also appeared four automorphisms, namely ${\rm Id}, \lambda_f, \Ad f, \Ad f \circ \lambda_f$.  Their
extensions send $U$ to $U$, $U^*$, $fUf$, and $fU^*f$  respectively. In fact, the images of $U$ through the remaining six 
proper endomorphisms are less easily guessed. We now tune up a method for recovering
those six endomorphisms from a somewhat different perspective, thus paving the way to a general analysis that will be carried out later on.\\

We start with case $\tilde{U} = U^2$. Since we have
$$U^2 S_1 S_1 = S_1 S_2 U, \quad U^2 S_1 S_2 = S_1 S_1, \quad U^2 S_2 S_1 = S_2 S_2 U, \quad U^2 S_2 S_2 = S_2 S_1 \ , $$
the multiplication by $U^2$ yields a monomial in $\CO_2$ only on the above elements of the form $S_i S_2$, $i=1,2$. We say that the monomials $S_i S_2$ are well-suited (for $U^2$).
Because $uS_i$ is a linear combination of $S_jS_hS_k^*$ and $U^2 S_r = S_r U$ we easily see that if Equation \eqref{ext1} is satisfied then,  Equation \eqref{ext2} is automatically
satisfied as well. Accordingly, all we have to bother with is
Equation \eqref{ext1}. The $u$'s below fulfil all requirements:
\begin{align*}
& u_{23} := S_1 S_2 (S_2 S_1)^* + S_1 S_1 (S_1 S_1)^* + S_2 S_2 (S_2 S_2)^* + S_2 S_1 (S_1 S_2)^* \equiv F\\
& u_{1342} := S_1 S_2 (S_2 S_2)^* + S_1 S_1 (S_1 S_2)^* + S_2 S_2 (S_2 S_1)^* + S_2 S_1 (S_1 S_1)^*
\end{align*}
These $u$'s are obtained by means of the following scheme. We start by taking the first  well-suited monomial according to the lexicographic order. We then match it with a monomial of the form $S_2 S_i$. This can be done in two different ways, which explains why we end up with two different unitaries. The next summand is the image of the chosen well-suited monomial under multiplication by $U^2$, followed by the same matching monomial as in the first summand, but with the first index changed to 1. We finally apply the procedure to the remaining well-suited monomial, but using matching monomials different from those already used.\\

The same method continues to work for $\tilde{U} = U^{-2}$. Now we have
$$U^{-2} S_1 S_1 = S_1 S_2, \quad U^{-2} S_1 S_2 = S_1 S_1 U^{-1}, \quad U^{-2} S_2 S_1 = S_2 S_2, \quad U^{-2} S_2 S_2 = S_2 S_1 U^{-1} \  $$
which means the well-suited monomial for $U^{-2}$ are those of the form $S_i S_1$, which leads to 
\begin{align*}
& u_{1243} = S_1 S_1 (S_2 S_1)^* + S_1 S_2 (S_1 S_1)^* + S_2 S_1 (S_2 S_2)^* + S_2 S_2 (S_1 S_2)^* \\
& u_{14} = S_1 S_1 (S_2 S_2)^* + S_1 S_1 (S_1 S_2)^* + S_2 S_1 (S_2 S_1)^* + S_2 S_2 (S_1 S_1)^* 
\end{align*}

\medskip
All is left to do now is treat the  case when  either $\tilde{U} = U^2 P_2 + U^{-2} P_1$ or $\tilde{U}= U^2P_1+U^{-2}P_2$, which can be thought of as a mixed case, so to speak.

\medskip
We start with  $\tilde{U} = U^2 P_2 + U^{-2} P_1$. In order to satisfy equation \ref{ext1}, we follow a similar method by merging the two cases above, namely we pick the only suited monomial for $U^2$ starting with $S_2$, that is $S_2 S_2$, thus determining the first two summands and then we pick the only well-suited monomial for $U^{-2}$ starting with $S_1$, that is $S_1 S_1$, determining the remaining two summands. All in all, we obtain
\begin{align*}
& u_{134} = S_2 S_2 (S_2 S_1)^* + S_2 S_1 (S_1 S_1)^* + S_1 S_1 (S_2 S_2)^* + S_1 S_2 (S_1 S_2)^* \\
& u_{123} = S_2 S_2 (S_2 S_2)^* + S_2 S_1 (S_1 S_2)^* + S_1 S_1 (S_2 S_1)^* + S_1 S_2 (S_1 S_1)^*  \ .
\end{align*}
This time, though, it is no longer clear that Equation \ref{ext2} is automatically satisfied. In fact, it is fulfilled only with $u_{123}$, as follows by direct computation.

\medskip
Now $\tilde{U}= U^2P_1+U^{-2}P_2$ can be dealt with in much the same way as above. By repeating the same scheme but picking the well-suited monomials $S_1 S_2$ for $U^2$ and $S_2 S_1$ for $U^{-2}$ we obtain
\begin{align*}
& u_{243}  = S_1 S_2 (S_2 S_1)^* + S_1 S_1 (S_1 S_1)^* + S_2 S_1 (S_2 S_2)^* + S_2 S_2 (S_1 S_2)^*  \\
& u_{142} = S_1 S_2 (S_2 S_2)^* + S_1 S_1 (S_1 S_2)^* + S_2 S_1 (S_2 S_1)^* + S_2 S_2 (S_1 S_1)^* \  ,
\end{align*}
but  only $u_{243}$  also satisfies Equation \ref{ext2}.\\

Note that $\lambda_f \lambda_{u_{123}} = \lambda_{u_{123}}$ and $\lambda_f \lambda_{u_{243}} = \lambda_{u_{243}}$ and therefore $\tilde\lambda_f \tilde\lambda_{u_{123}} = \tilde\lambda_{u_{123}}$ and $\tilde\lambda_f \tilde\lambda_{u_{243}} = \tilde\lambda_{u_{243}}$ , since $\tilde\lambda_f(U^{\pm 2} P_2 + U^{\mp 2} P_1) = U^{\pm 2} P_2 + U^{\mp 2} P_1$. In particular, $\lambda_{u_{123}}$, $\lambda_{u_{243}}$ and their extensions are proper endomorphisms.

\subsection{Case $\CP_2^3$}

The same technique would in fact apply to  endomorphisms $\lambda_u$ coming from a permutation matrix
$u \in \CP_2^3 \subset \CF_2^3$ too. Corresponding to $\tilde{U} = U^4$ or $U^{-4}$, it would  now yield 24 extensions each, which  are all proper endomorphisms of $\CQ_2$, 
cf. \cite[Proposition 6.1]{ACR}, as well as being proper endomorphisms at the level of the Cuntz algebra $\CO_2$ also, cf. \cite{CoSz11}.
Furthermore, for each of the four mixed cases $\tilde{U} = P_1 U^4 + P_2 U^{-4}$, $\tilde{U} = P_2 U^4 + P_1 U^{-4}$, $\tilde{U} = \varphi(P_1) U^4 + \varphi(P_2) U^{-4}$, and 
$\tilde{U} = \varphi(P_2) U^4 + \varphi(P_1) U^{-4}$  the technique would also yield another 4 extensions.
At any rate, we may as well refrain from describing the computations in detail here since the technique will be discussed in fuller generality below, where 
extendible permutative unitaries in $P_2^k$  will be found  aplenty for any $k\in\N$.
A complete list of the endomorphisms thus spotted, however, is provided in the appendix.

\subsection{Case $\CP_2^k$}

Of course, at each level $k$ we will recover those already obtained at lower levels and possibly more. In order to see that indeed we always find new extendible endomorphisms,
it is enough to realize that a similar method applies to $\CP_2^k$, i.e. permutations of $2^k$ objects, at least when $\tilde{U} = U^{2^{k-1}}$ or $\tilde{U} = U^{-2^{k-1}}$,
providing $2^{k-1}!$ different permutative endomorphisms for each case, and thus $2^{k-1}! \cdot 2$ new extendible (proper) endomorphisms. 
Of course, there are also the $2^{k-1}! 2 $ inner perturbations of the identity and of the flip-flop, all of which trivially extend to automorphisms.
Furthermore, there are those of the form 
$\Ad (v) \circ \varphi$ and $\Ad (v) \circ \varphi \circ \lambda_f$, with $v \in \CP_2^{k-1}$, again $2^{k-1}! 2$
(for $k > 2$ they are different from those above).

\begin{lemma}\label{wellsuitedk}
There are $2^{k-1}$ well-suited monomials for $U^{2^{k-1}}$, which are those of length $k$ ending in $S_2$, i.e.  of the form $S_{i_1}S_{i_2}\ldots S_{i_{k-1}}S_2$
with $i_1,i_2,\ldots i_{k-1}\in \{1,2\}$.
Likewise, there are $2^{k-1}$ well-suited monomials for $U^{-2^{k-1}}$, which are those of length $k$ ending in $S_1$, i.e.  of the form $S_{i_1}S_{i_2}\ldots S_{i_{k-1}}S_1$
with $i_1,i_2,\ldots i_{k-1}\in\{1,2\}$.
\end{lemma}

\begin{proof}
Indeed, it is enough to note that  $U^{2^{k-1}} S_{i_1}S_{i_2}\ldots S_{i_{k-1}} S_2= S_{i_1}S_{i_2}\ldots S_{i_{k-1}} S_1$ and
$U^{-2^{k-1}} S_{i_1}S_{i_2}\ldots S_{i_{k-1}} S_1= S_{i_1}S_{i_2}\ldots S_{i_{k-1}} S_2$.
\end{proof}

\begin{thm}
Let $p$ be a permutation of the set $W_2^{k-1}$. Consider the unitaries $u_p^{\pm}\in \CP_2^k$ given by
$$u_p^+=\sum_{\mu\in W_2^{k-1}} (S_\mu S_2)(S_2 S_{p(\mu)})^*+(S_\mu S_1)(S_1S_{p(\mu)})^*=\sum_{\mu\in W_2^{k-1},\,  i \in W_2^1} S_\mu S_i(S_i S_{p(\mu)})^*$$
$$u_p^-=\sum_{\mu\in W_2^{k-1}} (S_\mu S_1)(S_2 S_{p(\mu)})^*+(S_\mu S_2)(S_1S_{p(\mu)})^*$$
Then $\lambda_{u_p^{\pm}}$ both extend as endomorphisms of $\CQ_2$ with $\tilde\lambda_{u_p^{\pm}}(U)=U^{\pm 2^{k-1}}$.
\end{thm}

\begin{proof}
We shall only deal with $u_p^+$, for $u_p^-$ can be handled in exactly the same way.
We set $\tilde S_i\doteq u_p^+ S_i$ and $\tilde U\doteq U^{2^{k-1}}$. We need to make sure that
both $\tilde S_1=\tilde U\tilde S_2$ and $\tilde S_2\tilde U=\tilde U\tilde S_1$ hold true. But again, since 
$\tilde U$ is $U^{2^{k-1}}$, the latter is automatically satisfied provided that the former is.
Now on the one hand we have
$$\tilde S_1= \sum_{\mu\in W_2^{k-1}} \big((S_\mu S_2)(S_2 S_{p(\mu)})^*+(S_\mu S_1)(S_1S_{p(\mu)})^*\big)S_1=\sum_{\mu\in W_2^{k-1}}S_\mu S_1 S_{p(\mu)}^* $$
but on the other hand
\begin{align*}
\tilde U\tilde S_2 & =U^{2^{k-1}}u_p^+S_2=U^{2^{k-1}}\sum_{\mu\in W_2^{k-1}} \big((S_\mu S_2)(S_2 S_{p(\mu)})^*+(S_\mu S_1)(S_1S_{p(\mu)})^*\big)S_2\\
&=\sum_{\mu\in W_2^{k-1}} S_\mu US_2S_{p(\mu)}^*=\sum_{\mu\in W_2^k} S_\mu S_1 S_{p(\mu)}^*
\end{align*}
and so the equality is certainly satisfied.
\end{proof}

\begin{remark}
It is worth pointing out that  the equality $u_p^-(k)f=u_p^+(k)$ is satisfied for every integer $k$ and ever permutation $p$.
In general, it is not true that either $u_p^+(k)$ or $u_p^-(k)$ is fixed by $\lambda_f$. However, straightforward computations
show that this is certainly the case for both if the permutation $p$ commutes with $c$, where $c$ is the permutation on $W_2^k$ that swaps 
$1$ and $2$.
\end{remark}

\medskip
\begin{example}
We discuss a simple example.
For every integer $k$, let $F_k$ be the unitary in $\CF_2^{k+1}$ implementing the endomorphism $\varphi^k$.
In particular, $F_1= F$ with $F=\sum_{i,j=1,2}S_iS_jS_i^*S_j^*$.
We will show that if we take $p={\rm id}$ then $u_p^+(k) = F_{k-1}$, i.e. $\lambda_{u_p^+(k)} = \varphi^{k-1}$.
In this case $u_p^+(k)=\sum_{\mu\in W_2^{k-1},\,  i \in W_2^1} S_\mu S_i(S_i S_{\mu})^*$, which for $k=2$ gives
$u_p^+(2)=\sum_{j,i=1,2} S_j S_i S_j^*S_i^*=F= F_1$. 
We next prove by induction that the formula is true for every $k$. Supposing we have proved that $\lambda_{u_p^+(k)} = \varphi^{k-1}$,
we need to show that then $\lambda_{u_p^+(k+1)} = \varphi^{k}$ as well. But now $\varphi^k=\varphi\circ\varphi^{k-1}=\varphi\circ\lambda_{u_p^+(k)}=
\lambda_F\circ\lambda_{u_p^+(k)}=\lambda_{\varphi(u_p^+(k))F}$, so all we have to do is compute $\varphi(u_p^+(k))F$, which we do below:

\begin{align*}
\varphi(u_p^+(k))F&=\sum_{\mu,i,j} S_j S_\mu S_i(S_iS_\mu)^*S_j^*F=\sum_{\mu,i,j} S_j S_\mu S_iS_\mu^* S_i^* S_j^*\sum_{l,m=1,2}S_lS_mS_l^*S_m^*\\
&=\sum_{\mu,i,j} S_jS_\mu S_i S_\mu^*S_j^*S_i^*=\sum_{\mu,i,j } S_j S_\mu S_i(S_iS_jS_\mu)^*=u_p^+(k+1)
\end{align*}

It is now natural to expect 
the equality $\lambda_{u_p^-(k)} = \varphi^{k-1} \circ \lambda_f=
\lambda_f\circ\varphi^{k-1}$ to hold if $p=\rm{id}$.
Indeed, $\lambda_f\circ\varphi^{k}=\lambda_f\circ\lambda_{u_{\rm{id}}^+(k+1)}=\lambda_{\lambda_f(u_{\rm{id}}^+(k+1))f}$.
But $\lambda_f(u_{\rm{id}}^+(k+1))f$ is easily computed as follows
\begin{align*}
\lambda_f(u_{\rm{id}}^+(k+1))f &= \lambda_f\left(\sum_{\mu\in W_2^k}S_\mu S_2(S_2 S_\mu)^*+S_\mu S_1(S_1S_\mu)^*\right)f \\
& =\left(\sum_{\mu\in W_2^k}S_\mu S_1(S_1 S_\mu)^*+S_\mu S_2(S_2S_\mu)^*\right)f\\
&=\left(\sum_{\mu\in W_2^k}S_\mu S_1(S_1 S_\mu)^*+S_\mu S_2(S_2S_\mu)^*\right)\left(S_2S_1^*+S_1S_2^*\right)\\
&=\sum_{\mu\in W_2^k} S_\mu S_1(S_2S_\mu)^*+ S_\mu S_2 (S_1S_\mu)^*= u_{\rm{id}}^-(k+1)
\end{align*}

\end{example}

Of course, for every $p \in \CP_2^{k-1}$, the automorphisms  $\Ad(p)$ and $\Ad(p) \circ \lambda_f$ of $\CO_2$ extend to automorphisms of $\CQ_2$, with $\tilde{U}$ equal to $p U p^*$ and $p U^* p$, respectively. It follows at once by the results in \cite{ACR} that for each $\tilde{U}$ of the above form the $u \in \CP_2^k$ such that $\lambda_u$ extends with $\tilde\lambda_u(U)=\tilde{U}$
is necessarily unique and is given by $p\varphi(p^*)$ and $p\varphi(p^*)f$, respectively.

\medskip

We now deal with the case $\tilde U =P_1U^{2^{k-1}}+P_2U^{-2^{k-1}}$. Again, we can adopt the same strategy as before.
The well-suited monomials for $P_1U^{2^{k-1}}$ are obviously those of the form $S_1S_\alpha S_2$, where $\alpha\in W_2^{k-2}$ is any multi-index
of length $k-2$. In this way we get the following $(2^{k-2})!$ sums

$$\sum_{\alpha\in W_2^{k-2}} S_1S_\alpha S_2(S_2S_1S_{\sigma_1(\alpha)})^*+ S_1 S_\alpha S_1(S_1S_1S_{\sigma_1(\alpha)})^*$$

where $\sigma_1$ is a permutation on the set $W_2^{k-2}$. The well-suited monomials for $P_2U^{-2^{k-1}}$ are those of the form $S_2 S_\beta S_1$, where $\beta$ is any multi-index
of length $k-2$. As above, these give yield the following $(2^{k-2})!$ sums 

$$\sum_{\beta\in W_2^{k-2}} S_2S_\beta S_1(S_2S_2 S_{\sigma_2(\beta)})^*+ S_2 S_\beta S_2(S_1 S_2 S_{\sigma_2(\beta)})^*$$

where $\sigma_2$ is another (possibly different from $\sigma_1$) permutation on the set $W_2^{k-2}$. If we now combine the two sums, we finally obtain 
$(2^{k-2}!)^2$ unitaries in the Cuntz algebra given by

\begin{align*}
u_{\sigma_1,\sigma_2}\doteq&\sum_{\alpha, \beta\in W_2^{k-2}}S_1S_\alpha S_2(S_2S_1S_{\sigma_1(\alpha)})^*+ S_1 S_\alpha S_1(S_1S_1S_{\sigma_1(\alpha)})^*+\\
& \qquad\quad + S_2S_\beta S_1(S_2S_2 S_{\sigma_2(\beta)})^*+ S_2 S_\beta S_2(S_1 S_2 S_{\sigma_2(\beta)})^*
\end{align*}

\medskip

The case $\tilde U=P_1U^{-2^{k-1}}+P_2U^{2^{k-1}}$ is dealt with in pretty much the same way, and we find the following
$(2^{k-2}!)^2$ unitaries

\begin{align*}
u_{\sigma_1,\sigma_2}\doteq&\sum_{\alpha, \beta\in W_2^{k-2}}S_1S_\alpha S_1(S_2S_1S_{\sigma_1(\alpha)})^*+ S_1 S_\alpha S_2(S_1S_1S_{\sigma_1(\alpha)})^*+\\
&  \qquad\quad + S_2S_\beta S_2(S_2S_2 S_{\sigma_2(\beta)})^*+ S_2 S_\beta S_1(S_1 S_2 S_{\sigma_2(\beta)})^*
\end{align*}

We now come to the case $\tilde U=\varphi^h(P_1)U^{2^{k-1}}+\varphi^h(P_2)U^{-2^{k-1}}$, with
$h=1,2,\ldots k-2$ (the case $k=0$ has been adressed above). Since $\varphi(P_i)=\sum_{\alpha\in W_2^h} S_\alpha S_iS_i^*S_\alpha^*=\sum_{\alpha\in W_2^h} S_{\alpha i} S_{\alpha i}^*$,
the well-suited monomials for $\varphi^h(P_1)U^{2^{k-1}}$ are those of the form $S_{\alpha 1\beta 2}$ and 
the well-suited monomials for $\varphi^h(P_2)U^{-2^{k-1}}$ are those of the form $S_{\alpha 2 \beta 1}$, where $\beta$ is any multi-index of length $k-h-2$.
Bringing the pieces together, we finally get the following $N_{k,h}\doteq (2^{k-h-2}! 2^h)^2$ unitaries
$$
u_\sigma \doteq \sum_{\substack{ \alpha_1, \alpha_2\in W_2^h\; \\ \beta,\gamma\in W_2^{k-h-2}}}  S_{\alpha_1 1\beta 2} S_{2 \alpha_1 1 \sigma_{\alpha_1}(\beta)}^*+
S_{\alpha_1 1 \beta 1} S_{1\alpha_1 1\sigma_{\alpha_1}(\beta)}^*+S_{\alpha_2 2 \gamma 1}S_{2\alpha_2 2 \sigma_{\alpha_2}(\gamma)}^*+
S_{\alpha_2 2 \gamma 2}S_{1\alpha_ 2 2 \sigma_{\alpha_2}(\gamma)}^*
$$
where $\sigma$ is the set $\{\sigma_{\alpha_1}, \sigma_{\alpha_ 2}: \alpha_1,\alpha_2 \in W_2^h\}$, with $\sigma_{\alpha_i}$ being
a permutation of the set $W_2^{k-h-2}$ for each $\alpha$ and $i=1,2$. There follow the necessary computations to make sure
that the endomorphisms $\lambda_{u_\sigma}\in\End(\CO_2)$ actually extend to $\CQ_2$.

$$
\tilde S_1 = u_\sigma S_1=  \sum_{\substack{ \alpha_1, \alpha_2 \in W_2^h\; \\ \beta,\gamma\in W_2^{k-h-2}}} S_{\alpha_1 1\beta 1} S_{\alpha_1 1 \sigma_{\alpha_1}(\beta)}^*
+S_{\alpha_2 2\gamma 2}S_{\alpha_2 2\sigma_{\alpha_2}(\gamma)}^* 
$$

$$
\tilde S_2 = u_\sigma S_2= \sum_{\substack{ \alpha_1, \alpha_2\in W_2^h\; \\ \beta,\gamma\in W_2^{k-h-2}}} S_{\alpha_1 1\beta 2} S_{\alpha_1 1 \sigma_{\alpha_1}(\beta)}^* +S_{\alpha_2 2\gamma 1}S_{\alpha_2 2\sigma_{\alpha_2}(\gamma)}^* 
$$

\begin{align*}
\tilde U \tilde S_2 & = (U^{2^{k-1}}\varphi^h(P_1)+U^{-2^{k-1}}\varphi^h(P_2)) \left(\sum_{\substack{ \alpha_1, \alpha_2\in W_2^h\; \\ \beta,\gamma\in W_2^{k-h-2}}} S_{\alpha_1 1\beta 2} S_{\alpha_1 1 \sigma_{\alpha_1}(\beta)}^* +S_{\alpha_2 2\gamma 1}S_{\alpha_2 2\sigma_{\alpha_2}(\gamma)}^* \right)\\
& = \sum_{\substack{ \alpha_1,\alpha_2\in W_2^h\; \\ \beta,\gamma\in W_2^{k-h-2}}} U^{2^{k-1}} S_{\alpha_1 1\beta 2} S_{\alpha_1 1 \sigma_{\alpha_1}(\beta)}^*+ U^{-2^{k-1}}S_{\alpha_2 2\gamma 1}S_{\alpha_2 2\sigma_{\alpha_2}(\gamma)}^* \\
& =  \sum_{\substack{ \alpha_1,\alpha_2 \in W_2^h\; \\ \beta,\gamma\in W_2^{k-h-2}}} S_{\alpha_1 1\beta 1} S_{\alpha_1 1 \sigma_{\alpha_1}(\beta)}^*+S_{\alpha_2 2\gamma 2}S_{\alpha_2 2\sigma_{\alpha_2}(\gamma)}^* =\tilde S_1
\end{align*}

\begin{align*}
\tilde U \tilde S_1 & = (U^{2^{k-1}}\varphi^h(P_1)+U^{-2^{k-1}}\varphi^h(P_2)) \left(\sum_{\substack{ \alpha_1, \alpha_2 \in W_2^h\; \\ \beta,\gamma\in W_2^{k-h-2}}} S_{\alpha_1 1\beta 1} S_{\alpha_1 1 \sigma_{\alpha_1}(\beta)}^* +S_{\alpha_2 2\gamma 2}S_{\alpha_2 2\sigma_{\alpha_2}(\gamma)}^* \right)\\
& = \sum_{\substack{ \alpha_1,\alpha_2\in W_2^h\; \\ \beta,\gamma\in W_2^{k-h-2}}} U^{2^{k-1}} S_{\alpha_1 1\beta 1} S_{\alpha_1 1 \sigma_{\alpha_1}(\beta)}^*+U^{-2^{k-1}}S_{\alpha_2 2\gamma 2}S_{\alpha_2 2\sigma_{\alpha_2}(\gamma)}^*\\
& = \sum_{\substack{ \alpha_1,\alpha_2\in W_2^h\; \\ \beta,\gamma\in W_2^{k-h-2}}} S_{\alpha_11\beta} US_1 S_{\alpha_1 1 \sigma_{\alpha_1}(\beta)}^*+S_{\alpha_2 2\gamma} U^* S_2 S_{\alpha_2 2\sigma_{\alpha_2}(\gamma)}^*\\
\end{align*}

\begin{align*}
\tilde S_2 \tilde U  & = \left( \sum_{\substack{ \alpha_1, \alpha_2\in W_2^h\; \\ \beta,\gamma\in W_2^{k-h-2}}} S_{\alpha_1 1\beta 2} S_{\alpha_1 1 \sigma_{\alpha_1}(\beta)}^* +S_{\alpha_2 2\gamma 1}S_{\alpha_2 2\sigma_{\alpha_2}(\gamma)}^*  \right) (U^{2^{k-1}}\varphi^h(P_1)+U^{-2^{k-1}}\varphi^h(P_2)) \\
 & = \sum_{\substack{ \alpha_1,\alpha_2\in W_2^h\; \\ \beta,\gamma\in W_2^{k-h-2}}} S_{\alpha_1 1\beta 2} S_{\alpha_1 1 \sigma_{\alpha_1}(\beta)}^*U^{2^{k-1}}+S_{\alpha_2 2\gamma 1}S_{\alpha_2 2\sigma_{\alpha_2}(\gamma)}^*  U^{-2^{k-1}} \\ 
 & = \sum_{\substack{ \alpha_1,\alpha_2\in W_2^h\; \\ \beta,\gamma\in W_2^{k-h-2}}} S_{\alpha_1 1\beta 2} U S_{\alpha_1 1 \sigma_{\alpha_1}(\beta)}^*+S_{\alpha_2 2\gamma 1}U^*S_{\alpha_2 2\sigma_{\alpha_2}(\gamma)}^*   \\
\end{align*}

\bigskip

The case $\tilde U=\varphi^h(P_2)U^{2^{k-1}}+\varphi(P_1)^hU^{-2^{k-1}}$ can be dealt with in pretty much the same way.
The formula thus got to reads as follows:

$$
u_\sigma \doteq \sum_{\substack{ \alpha_1,\alpha_2\in W_2^h\; \\ \beta,\gamma\in W_2^{k-h-2}}}  S_{\alpha_1 2\beta 2} S_{2 \alpha_1 2 \sigma_{\alpha_1}(\beta)}^*+
S_{\alpha_1 2 \beta 1} S_{1\alpha_1 2\sigma_{\alpha_1}(\beta)}^*+S_{\alpha_2 1\gamma 1}S_{2\alpha_2 1\sigma_{\alpha_2}(\gamma)}^*+
S_{\alpha_2 1 \gamma 2}S_{1\alpha_2 1\sigma_{\alpha_2}(\gamma)}^*
$$
where $\sigma$ is the set $\{\sigma_{\alpha_1}, \sigma_{\alpha_ 2}: \alpha_1,\alpha_2 \in W_2^h\}$, with $\sigma_{\alpha_i}$ being
a permutation of the set $W_2^{k-h-2}$ for each $\alpha_i$ and $i=1,2$.
In particular, we still have $N_{k,h}\doteq (2^{k-h-2}! 2^h)^2$ extendible endomorphisms.
For the reader's convenience, we can finally sum up our findings in the following statement. 

\begin{thm}\label{listun}
Given a natural number $k\geq 2$, for each $h\in\{0,1,2,\ldots, k-2\}$ let $N_{k,h}$ be $(2^{k-h-2}! 2^h)^2$.

\medskip
First, there are $N_{k,h}$ unitaries $u_\sigma^{(1)}$ in  $\CP_2^k$ such that the associated endomorphism $\lambda_{u_\sigma^{(1)}}$  is extendible  
with  $\widetilde{\lambda_{u_\sigma^{(1)}}}(U)=\varphi^h(P_1)U^{2^{k-1}}+\varphi^h(P_2)U^{-2^{k-1}}$, and these are given by 
$$
u_\sigma^{(1)} \doteq  \sum_{\substack{ \alpha_1, \alpha_2\in W_2^h\; \\ \beta,\gamma\in W_2^{k-h-2}}}  S_{\alpha_1 1\beta 2} S_{2 \alpha_1 1 \sigma_{\alpha_1}(\beta)}^*+
S_{\alpha_1 1 \beta 1} S_{1\alpha_1 1\sigma_{\alpha_1}(\beta)}^*+S_{\alpha_2 2 \gamma 1}S_{2\alpha_2 2 \sigma_{\alpha_2}(\gamma)}^*+
S_{\alpha_2 2 \gamma 2}S_{1\alpha_ 2 2 \sigma_{\alpha_2}(\gamma)}^*$$
where, for any $\alpha_1, \alpha_2\in W_2^h$, $\sigma_{\alpha_1}$ and $\sigma_{\alpha_2}$  are permutations on  $W_2^{k-h-2}$.

\medskip
Second, there are $N_{k,h}$ unitaries $u_\sigma^{(2)}$ in  $\CP_2^k$ such that the associated endomorphism $\lambda_{u_\sigma^{(2)}}$  is extendible  
with  $\tilde{U}=\varphi^h(P_2)U^{2^{k-1}}+\varphi^h(P_1)U^{-2^{k-1}}$, and these are given by 

$$
u_\sigma^{(2)} \doteq
\sum_{\substack{ \alpha_1,\alpha_2\in W_2^h\; \\ \beta,\gamma\in W_2^{k-h-2}}}  S_{\alpha_1 2\beta 2} S_{2 \alpha_1 2 \sigma_{\alpha_1}(\beta)}^*+
S_{\alpha_1 2 \beta 1} S_{1\alpha_1 2\sigma_{\alpha_1}(\beta)}^*+S_{\alpha_2 1\gamma 1}S_{2\alpha_2 1\sigma_{\alpha_2}(\gamma)}^*+
S_{\alpha_2 1 \gamma 2}S_{1\alpha_2 1\sigma_{\alpha_2}(\gamma)}^*
$$

with the same notation as in the first case.

\end{thm}

\medskip
\begin{remark}
It goes without saying that there might be more possibilities for $\tilde{U}$, other than those considered above, that still give rise to extendible endomorphisms. In addition, already with the values of $\tilde{U}$ we have considered, there is no evidence that the unitaries listed in Theorem \ref{listun} actually exhaust all possible cases.   
\end{remark}

\begin{remark}
All of the endomorphisms we have produced  so far are clearly endomorphisms of the Bunce-Deddens algebra $\CQ_2^{\IT}$ too.  
\end{remark}

\appendix
\section{Permutative endomorphisms of $\CQ_2$ at level  $3$}

The notation we adopt here is similar to that already used in \cite{ACR3}, where
monomials $S_\alpha S_\beta^*$ are denoted by $S_{\alpha_1,\alpha_2,\ldots, \alpha_{|\alpha|}, \beta_1,\beta_2,\ldots, \beta_{|\beta|}}$ for any $\alpha,\beta\in W_2$. There follows a table where the permutative unitaries in $\CP_2^k$ yielded
by the general method described in the previous section are shown along with the action of the corresponding endomorphisms on $U$.  

{\footnotesize
\[
\begin{array}{cc}
\\
\hline u\in \CP_2^3 & \widetilde{\lambda_u}(U)  \\
\hline
S_{112,211} + S_{111,111} + S_{122,212} + S_{121,112}+ S_{212,221} + S_{211,121} + S_{222,222}+ S_{221,122}  &  U^4 \\
S_{112,211}+ S_{111,111}+ S_{122,212} + S_{121,112} + S_{212, 221}+ S_{211,121} + S_{222,221} + S_{221,121} &  U^4 \\
S_{112,211} + S_{111,111}+ S_{122,212} + S_{121,112} + S_{212,212}+ S_{211,112}+ S_{222,222}+ S_{221,122} &  U^4 \\
S_{112,211}+ S_{111,111} + S_{122,212}+ S_{121,112}+S_{212,212}+S_{211,112}+S_{222,212}+ S_{221,112}  &  U^4 \\
S_{112,211}+S_{111,111}+S_{122,222}+S_{121,122}+S_{212,212}+S_{211,112}+S_{222,221}+S_{221,121} &  U^4 \\
S_{112,211}+S_{111,111}+S_{122,222}+S_{121,122}+S_{212,221}+S_{211,121}+S_{222,212}+S_{221,112} &  U^4 \\
S_{112,212}+S_{111,112}+S_{122,211}+S_{121,111}+S_{212,221}+S_{211,121}+S_{222,222}+S_{221,122} &  U^4 \\
S_{112,212}+S_{111,112}+S_{122,211}+S_{121,111}+S_{212,222}+S_{211,122}+S_{222,221}+S_{221,121} &  U^4 \\
S_{112,212}+S_{111,112}+S_{122,211}+S_{121,121}+S_{212,211}+S_{211,111}+S_{222,222}+S_{221,122} &  U^4 \\
S_{112,212}+S_{111,112}+S_{122,211}+S_{121,121}+S_{212,222}+S_{211,122}+S_{222,211}+S_{221,111} &  U^4 \\
S_{112,212}+S_{111,112}+S_{122,222}+S_{121,122}+S_{212,211}+S_{211,111}+S_{222,221}+S_{221,121} &  U^4 \\
S_{112,212}+S_{111,112}+S_{122,222}+S_{121,122}+S_{212,221}+S_{211,121}+S_{222,211}+S_{221,111} &  U^4 \\
S_{112,221}+S_{111,121}+S_{122,211}+S_{121,111}+S_{212,212}+S_{211,112}+S_{222,222}+S_{221,122} &  U^4 \\
S_{112,221}+S_{111,121}+S_{122,211}+S_{121,111}+S_{212,222}+S_{211,122}+S_{222,212}+S_{221,112} &  U^4 \\
S_{112,221}+S_{111,121}+S_{122,212}+S_{121,112}+S_{212,211}+S_{211,111}+S_{222,222}+S_{221,122} &  U^4 \\
S_{112,221}+S_{111,121}+S_{122,212}+S_{121,112}+S_{212,222}+S_{211,122}+S_{222,211}+S_{221,111} &  U^4 \\
S_{112,221}+S_{111,121}+S_{122,222}+S_{121,122}+S_{212,211}+S_{211,111}+S_{222,212}+S_{221,112} &  U^4 \\
S_{112,221}+S_{111,121}+S_{122,222}+S_{121,122}+S_{212,212}+S_{211,112}+S_{222,211}+S_{221,111} &  U^4 \\
S_{112,222}+S_{111,122}+S_{122,211}+S_{121,111}+S_{212,212}+S_{211,112}+S_{222,221}+S_{221,121} &  U^4 \\
S_{112,222}+S_{111,122}+S_{122,211}+S_{121,111}+S_{212,221}+S_{211,121}+S_{222,212}+S_{221,112} &  U^4 \\
S_{112,222}+S_{111,122}+S_{122,212}+S_{121,112}+S_{212,211}+S_{211,111}+S_{222,221}+S_{221,121} &  U^4 \\
S_{112,222}+S_{111,122}+S_{122,212}+S_{121,112}+S_{212,221}+S_{211,121}+S_{222,211}+S_{221,111} &  U^4 \\
S_{112,222}+S_{111,122}+S_{122,221}+S_{121,121}+S_{212,211}+S_{211,111}+S_{222,212}+S_{221,112} &  U^4 \\
S_{112,222}+S_{111,122}+S_{122,221}+S_{121,121}+S_{212,212}+S_{211,112}+S_{222,211}+S_{221,111} &  U^4 \\
S_{112,211}+S_{111,111}+S_{122,212}+S_{121,112}+S_{211,221}+S_{212,121}+S_{221,222}+S_{222,122} &  P_1U^4+P_2U^{-4} \\
S_{112,211}+S_{111,111}+S_{122,212}+S_{121,112}+S_{211,222}+S_{212,122}+S_{221,221}+S_{222,121} &  P_1U^4+P_2U^{-4} \\
S_{112,212}+S_{111,112}+S_{122,211}+S_{121,111}+S_{211,221}+S_{212,121}+S_{221,222}+S_{222,122} & P_1U^4+P_2U^{-4} \\
S_{112,212}+S_{111,112}+S_{122,211}+S_{121,111}+S_{211,222}+S_{212,122}+S_{221,221}+S_{222,121} & P_1U^4+P_2U^{-4} \\
S_{212,221}+S_{211,121}+S_{222,222}+S_{221,122}+S_{111,211}+S_{112,111}+S_{121,212}+S_{122,112} & P_1U^{-4}+P_2U^{4} \\
S_{212,221}+S_{211,121}+S_{222,222}+S_{221,122}+S_{111,212}+S_{112,112}+S_{121,211}+S_{122,111} & P_1U^{-4}+P_2U^{4} \\
S_{212,222}+S_{211,122}+S_{222,221}+S_{221,121}+S_{111,211}+S_{112,111}+S_{121,212}+S_{122,112} & P_1U^{-4}+P_2U^{4} \\
S_{212,222}+S_{211,122}+S_{222,221}+S_{221,121}+S_{111,212}+S_{112,112}+S_{121,211}+S_{122,111} & P_1U^{-4}+P_2U^{4} \\
S_{112,211}+S_{111,111}+S_{212,221}+S_{211,121}+S_{121,212}+S_{122,112}+S_{221,222}+S_{222,122} & \varphi(P_1)U^{4}+\varphi(P_2)U^{-4} \\
S_{112,211}+S_{111,111}+S_{212,221}+S_{211,121}+S_{121,222}+S_{122,122}+S_{221,212}+S_{222,112} & \varphi(P_1)U^{4}+\varphi(P_2)U^{-4} \\
S_{112,221}+S_{111,121}+S_{212,211}+S_{211,111}+S_{121,212}+S_{122,112}+S_{221,222}+S_{222,122} & \varphi(P_1)U^{4}+\varphi(P_2)U^{-4} \\
S_{112,221}+S_{111,121}+S_{212,211}+S_{211,111}+S_{121,222}+S_{122,122}+S_{221,212}+S_{222,112} & \varphi(P_1)U^{4}+\varphi(P_2)U^{-4} \\
S_{122,212}+S_{121,112}+S_{222,222}+S_{221,122}+S_{111,211}+S_{112,111}+S_{211,221}+S_{212,121} & \varphi(P_1)U^{-4}+\varphi(P_2)U^{4} \\
S_{122,212}+S_{121,112}+S_{222,222}+S_{221,122}+S_{111,221}+S_{112,121}+S_{211,211}+S_{212,111} & \varphi(P_1)U^{-4}+\varphi(P_2)U^{4} \\
S_{122,222}+S_{121,122}+S_{222,212}+S_{221,112}+S_{111,211}+S_{112,111}+S_{211,221}+S_{212,121} & \varphi(P_1)U^{-4}+\varphi(P_2)U^{4} \\
S_{122,222}+S_{121,122}+S_{222,212}+S_{221,112}+S_{111,221}+S_{112,121}+S_{211,211}+S_{212,111} & \varphi(P_1)U^{-4}+\varphi(P_2)U^{4} \\
\hline
\end{array}
\]
}

\end{document}